\documentclass[a4paper,12pt]{article}

\usepackage[utf8]{inputenc} % Encodage UTF-8
\usepackage{lmodern} % Latin Modern, compatible T1
\usepackage[english]{babel} % Langue du document
\usepackage{a4wide}         %diminue les marges

\usepackage{amsmath, amssymb, amsthm, amsfonts , stmaryrd} % Mathématiques avancées
\usepackage{mathrsfs}  % Pour les lettres en script (\mathscr)
\usepackage{bbold}     % Pour \mathbb{1}
\usepackage{bigints}   % Intégrales grandes
\usepackage{upgreek}   % Lettres grecques droites (\upalpha, etc.)
\usepackage{systeme}   % Systèmes d'équations
\usepackage{calc}      % Calculs dans les dimensions

\usepackage{graphicx}
\graphicspath{{./images/}} % Chemin pour les images
\usepackage[font=small,labelfont=bf]{caption} % Légendes améliorées
\usepackage{subcaption} % Sous-figures

\usepackage{hyperref}

\newtheoremstyle{custom_theorem} % Nom du style
  {10pt} % Espace avant
  {10pt} % Espace après
  {\itshape} % Style du texte principal
  {} % Indentation (vide = pas d'indentation)
  {\bfseries} % Style du titre
  {.} % Ponctuation après le titre
  { } % Espace entre le titre et le texte
  {} % Spécification du titre

\newtheoremstyle{custom_remark} % Nom du style
  {10pt} % Espace avant
  {10pt} % Espace après
  {\normalfont} % Style du texte principal
  {} % Indentation (vide = pas d'indentation)
  {\bfseries} % Style du titre
  {:} % Ponctuation après le titre
  { } % Espace entre le titre et le texte
  {} % Spécification du titre

\theoremstyle{custom_theorem}
\newtheorem{theorem}{Theorem}[section]
\newtheorem{proposition}[theorem]{Proposition}
\newtheorem{lemma}[theorem]{Lemma}
\newtheorem{corollary}[theorem]{Corollary}

\theoremstyle{custom_remark}
\newtheorem{remark}{Remark}[section]
\newtheorem{example}{Example}[section]
\newtheorem{definition}[theorem]{Definition}
\newtheorem{assumption}[theorem]{Assumption}

\numberwithin{equation}{section}

\DeclareMathOperator{\tr}{Tr}

\newcommand{\GCC}{\mathcal{GCC}}

\renewcommand{\P}{\mathbb{P}}

\newcommand{\E}{\mathbb{E}}
\newcommand{\N}{\mathbb{N}}
\newcommand{\1}{\mathbb{1}}
\newcommand{\C}{\mathbb{C}}

\renewcommand{\H}{\mathcal{H}}
\newcommand{\CC}{\mathcal{C}}
\newcommand{\M}{\mathcal{M}}

\newcommand{\A}{\mathcal{A}}
\newcommand{\B}{\mathcal{B}}

\newcommand{\G}{\mathcal{G}}

\newcommand{\dd}{\mathrm{d}}
\renewcommand{\Im}{\mathrm{Im}}
\renewcommand{\Re}{\mathrm{Re}}

\newcommand{\D}{\mathcal{D}}

\newcommand{\X}{\mathcal{X}}
\newcommand{\Y}{\mathcal{Y}}

\DeclareUnicodeCharacter{2212}{\ensuremath{-}}

\title{Convergence rate of the diagonal-valued Cauchy-transform for permutation invariant random matrices}
\author{Alexis Imbert \thanks{Universit\'{e} de Bordeaux, Institut de Math\'{e}matiques de Bordeaux, 351 Cours de la Lib\'{e}ration, 33400 Talence, France. Email : \href{mailto:alexis.imbert@math.u-bordeaux.fr}{alexis.imbert@math.u-bordeaux.fr}}}

\begin{document}

\maketitle

\begin{abstract}
    Let $A$ be a permutation invariant random matrix and $B$ another random matrix. We give a quantitative bound on the difference between the diagonal of the resolvent of $A+B$ and the diagonal of the resolvent of the free sum with amalgamation over the diagonal of $A$ and $B$. Moreover, we improve the rate of convergence whenever the matrices $A$ and $B$ are sparse and bounded in operator norm. Doing so, we explicitly construct the free sum over the diagonal of $A$ and $B$ as an adjacency operator of a weighted locally finite graph. 
\end{abstract}

\tableofcontents

\section{Introduction}
    Understanding the asymptotic spectrum of the sum of two random matrices is a classical problem in random matrix theory. From Voiculescu's pioneering work \cite{voiculescu_operations_2018}, we know that when one of the two matrices is invariant in law under conjugation by unitary matrices the two matrices are asymptotically free. In that scenario, provided the empirical spectral distribution (ESD) of the two matrices converge, the limiting measure of the sum is the free additive convolution of the marginal limiting measures. One of the main tool to prove this kind of convergence is the moment method which shows convergence of the Cauchy transform (or equivalently the normalized trace of the resolvent) toward some analytic function on the right domain. This proves in particular that the limit of the Cauchy transform is still the Cauchy transform of some measure that we can now characterize. Moreover, still in that scenario, provided some smoothness assumption on the limiting ESD and boundedness of the norm of the matrices, it has been shown that the convergence of the Cauchy transform of the sum toward its limit happens at speed $N^{-1}$ (where $N$ is the size of the matrix). The first result was obtained by Chatterjee in \cite{chatterjee_concentration_2007} which was later improved by Kargin in \cite{kargin_concentration_2012} and \cite{kargin_subordination_2015}. These results show a convergence not only of the Cauchy transform but of the diagonal of the resolvent. In a series of following paper, Bao, Erdos and Schnelli \cite{bao_convergence_2017,bao_local_2017} proved that the optimal speed was $N^{-1}$ in the bulk of the limiting measure.

In this article, we are interested in a more general case: when one of the two matrices is invariant in law by conjugation by permutation matrices. In that scenario, provided that the matrices converge in traffic distribution and verify some factorization property Male \cite{male_traffic_2020} proved that the random matrices at stake are asymptotically traffic free. Moreover, it has been shown in \cite{au_large_2021} that permutation invariant random matrices are asymptotically free over the diagonal. The latter is a particular case of operator-valued freeness \cite{mingo_free_2017,williams_analytic_2017,popa_non-commutative_2013}. Therefore, the asymptotic joint traffic (resp. diagonal) distribution of the two matrices can be recovered just from their marginal traffic (resp. diagonal) distribution (see Definition \ref{def: traffic distribution} for more details). Then, if we are interested only in the asymptotic ESD of the sum, since the joint traffic and diagonal distribution encompass the joint $*$-distribution, it is enough to know the marginal asymptotic traffic or diagonal distributions. For a single permutation invariant random matrix, and in particular for adjacency matrix of random regular graphs, several work has been done in the recent years, see for instance \cite{bauerschmidt_local_2017,bauerschmidt_local_2019}. 

Let $X$ be a matrix of size $N$, and let $D$ be a diagonal matrix whose entries all have positive imaginary part. The diagonal-valued Cauchy-transform of $X$ is defined as $G_{X}(D):=\Delta((D-X)^{-1})$ where for any matrix $M$, $\Delta(M):=(M_{i,j}\delta_{i,j})_{1\leq i,j\leq N}$. A more general definition is given in Section 2. One knows from the scalar case that it is customary to study the Cauchy transform of a measure to recover some information about it. The same goes for operator-valued Cauchy-transform (see \cite{popa_non-commutative_2013,williams_analytic_2017,voiculescu_operations_2018}). 

Provided $A$ and $B$ satisfy Assumptions \ref{H} and \ref{assumption frobenius} at speed $M$, Theorem \ref{thm A} states that for $|z|>r$,
$$\E\frac{1}{N}||G_{A+B}(zI)-G_{a+b}(zI)||_{F}^{2}\leq (r/|z|)^{2cM},$$
where $r,c$ are some constant depending on the Assumptions on $A$ and $B$ and $a+b$ is the free sum over the diagonal of $A$ and $B$ explicitly constructed in Section 3. Note that $M$ is at most of order $\log N/\log\log N$ so that the rate of convergence is at most $N^{-\beta/\log\log N}$.

Assumption \ref{H} controls the growth of traffic moments for graph monomials of size $M$ whereas Assumption \ref{assumption frobenius} controls the Frobenius norm of the $M$-th power of $A$ and $B$ (or equivalently, their Schatten-$M$ norm), where $M$ goes to infinity with $N$ and is at most of order $\log N/\log\log N$.

Assumption \ref{H} is satisfied for a wide range of matrix models such as the Wigner matrices, the heavy Wigner matrices that contain in particular the adjacency matrices of Erdos-Renyi graphs with parameter at least $d/N$ for some fixed  $0<d$ and sparse bounded matrices. Moreover, Assumption \ref{H} is stable by entry-wise product, (see the last paragraph of Section 4.2). It can be thought in the following way, given a graph with at most $M$ vertices and edges, we control the number of injective maps sending this graph to the adjacency graph of the matrices $A$ and $B$.

However, Assumption \ref{assumption frobenius} is more restrictive. Indeed, looking at adjacency matrices of Erdos-Renyi random graphs for instance, it applies only when the parameter $d$ also goes to infinity with $N$. Note however that it is more lenient than asking for the operator norm of the matrices to be uniformly bounded. Indeed, take again Erdos-Renyi graphs, when the parameter $d$ is of order smaller than $\log N$, it is shown in \cite{benaych-georges_largest_2019} that the operator norm of the adjacency matrix of Erdos-Renyi random graphs with such parameter goes to infinity with $N$. However, for $M$ small enough, that goes to infinity with $N$, Assumption \ref{assumption frobenius} is verified for such a model. Similarly, for Heavy-Wigner matrices, if the (properly normalized) moment of order $k$ of the entries is not zero for $k\geq3$ (see \cite{zakharevich_generalization_2006,male_limiting_2012}), Assumption \ref{assumption frobenius} cannot be verified at any speed $M$, see Section 4.3. We therefore restrict ourselves with what we call diluted Wigner matrices, where the condition on the moments of the entries is stated in Equation \eqref{moment condition diluted wigner}.

Theorem \ref{thm B} requires that the two matrices are uniformly sparse and bounded in operator norm, but controls the difference of the Cauchy-transform evaluated in any diagonal matrices and improves the rate of convergence to $N^{-\gamma}$ for some positive explicit constant $\gamma$. It can thus be applied to uniform $d$-regular graphs for $d$ bounded, continuing the work of \cite{bauerschmidt_local_2019} in the case of the addition of two such matrices. The case where matrices are both bounded in operator norm is easier to deal with because it allows us to control the rest of the series expansion of the resolvent. Note that it implies Assumption \ref{assumption frobenius}.

Both proofs rely on a moment method where Assumption \ref{assumption frobenius}, resp. the uniform bound in operator norm, allows us to control the rest in the series expansion of the resolvent for Theorem \ref{thm A}, resp. Theorem \ref{thm B}.

\paragraph{Organization of the paper}
In Section 2, we introduce the necessary tools from operator-valued free probability so that we are able to state Theorem \ref{thm A} and Theorem \ref{thm B}. In Section 3, we realize the free sum with amalgamation over the diagonal of $A$ and $B$ as an adjacency operator of a locally finite graph. Section 4 is dedicated to Assumptions \ref{H} and \ref{assumption frobenius} and to showing that many random matrix models verify those assumptions. In Section 5, we prove the two main theorems.

\section{Statement of the main Theorem}
    In this section, we provide the necessary background in operator-valued free probability for us to be able to state our main theorems. We specifically give a construction of the free sum with amalgamation over the diagonal of two random matrices. 
    \subsection{Notions of diagonal-valued free probability}
        We recall some basic notions of operator-valued free probability. One can look into \cite{mingo_free_2017, popa_non-commutative_2013, williams_analytic_2017} and the references therein for a more detailed introduction to operator-valued non commutative probabilities.

\begin{definition}
    A $C^{*}$-\emph{operator-valued probability space} is a triplet $(\A,\B,E)$ where $\A$ is a $C^{*}$-algebra, $\B\subset\A$ is a unital $C^{*}$-subalgebra of $\A$ and $E:\A\rightarrow\B$ is a completely positive, unital, linear map that satisfies the $\B$-bi-modularity condition: $E(b_{1}ab_{2})=b_{1}E(a)b_{2}$ for all $b_{1},b_{2}\in \B$ and all $a\in \A$. The map $E$ is then called a conditional expectation.

    We say that two elements $x,y\in\A$ are free over $\B$ (or $\B$-free, or free with amalgamation over $\B$) whenever the following statement holds. Let $\A_{1}\subset\A$, (resp.$\A_{2}\subset\A$) be the algebra generated by $x$ (resp. $y$) and $\B$. For all $n\geq1$ and $a_{1},\cdots,a_{n}$ such that they are 
    \begin{itemize}
        \item Centered: for all $1\leq i \leq n$, $E(a_{i})=0$,
        \item Alternating: for all $1\leq i \leq n$, $a_{i}\in\A_{j_{i}}$ with $j_{1}\neq j_{2},\cdots,j_{n-1}\neq j_{n}$, 
    \end{itemize}
    then 
    \begin{equation*}
        E(a_{1}a_{2}\cdots a_{n})=0.
    \end{equation*}
\end{definition}

Let $\X$ be a self-adjoint formal variable algebraically free from $\B$, we define $\B\langle\X\rangle$ as the $\B$-bimodule $*$-algebra of non-commutative polynomials over $\B$, that is the linear span of monomials of the form $b_{1}\X b_{2}\X\cdots\X b_{n}$ with $b_{i}\in\B$ for $1\leq i \leq n$, with $(b_{1}\X b_{2}\X\cdots\X b_{n})^{*}=b_{n}^{*}\X b_{n-1}^{*}\X\cdots\X b_{1}^{*}$.

\begin{definition}
    Let $a\in(\A,\B,E)$ as above. The $\B$-valued distribution (or $E$-distribution) of $a$ is the map
    \begin{align*}
        \mu_{a}:\,&\B\langle\X\rangle\rightarrow\B\\
        &P\mapsto E(P(a)),
    \end{align*}
    where $P(a)$ refers to the evaluation map of the non-commutative polynomial $P(\X)\in\B\langle\X\rangle$.
\end{definition}

We define an abstract set of distributions $\Sigma_{0}$ as the set of unital, positive $\B$-bimodular maps $\mu:\B\langle\X\rangle\rightarrow\B$, such that 
\begin{enumerate}
    \item For any $n\in\N$ and any non commutative polynomial $P_{1}(\X),\cdots,P_{n}(\X) \in\B\langle\X\rangle$, we have $\left[\mu(P_{i}(\X)^{*}P_{j}(\X))\right]_{1\leq i,j\leq n}\geq 0$ in $\M_{n}(\B)$,
    \item There exists $M>0$ such that for all $b_{1},\cdots,b_{n}\in \B$,
    \begin{equation*}
        ||\mu(b_{0}\X b_{1}\cdots\X b_{n})||\leq M^{n}||b_{0}||\cdots||b_{n}||.
    \end{equation*}
\end{enumerate}

The norm above is the norm inherited from the $C^{*}$-algebra, which is the operator norm whenever one deals with random matrices.

It turns out that these two conditions characterize the $\B$-valued distribution of an element. 

\begin{theorem}[Proposition 2.2 of \cite{popa_non-commutative_2013}]\label{caractérisation mesures algébriques}
    Let $\mu$ be a positive unital $\B$-bimodular map that satisfies the first point above. It satisfies the second point if and only if there exists a \break $\B$-valued non-commutative probability space $(\A,\B, E)$ and a self-adjoint element $a\in\A$ such that $\mu=\mu_{a}$.
\end{theorem}

As in the scalar case, it is possible to define some analytic transforms associated to a distribution $\mu$. Every element $x\in\A$ can uniquely be written as $\Re(x)+i\Im(x)$, where $\Re(x)=\frac{x+x^{*}}{2}$ and $\Im(x)=\frac{x-x^{*}}{2i}$. We denote $\mathbb{H}^{+}(\B):=\{b\in\B,\,\Im(b)>0\}$, where $y>0$ stands for $y\geq 0$ and $y$ invertible or equivalently, $y\geq \varepsilon 1_{\A}$ for some strictly positive real number $\varepsilon$. Elements in $\mathbb{H}^{+}(\B)$ are all invertible. Furthermore, for a self-adjoint element $x=x^{*}$ and an element $b\in\mathbb{H}^{+}(\B)$, $b-x$ is still invertible. Hence, the operator upper half-plane $\mathbb{H}^{+}(\B)$ is the appropriate domain to define the Cauchy transform of a self adjoint element $x=x^{*}$ 

Let $x$ be a self-adjoint element in an operator valued probability space $(\A,\B, E)$, 
\begin{itemize}
    \item The operator-valued Cauchy transform of $x$ is 
    \begin{equation}
        G_{x}(b):=E((b-x)^{-1})=\mu_{x}((b-\X)^{-1}).
    \end{equation}
    \item The reciprocal Cauchy transform of $x$ is 
    \begin{equation}
        F_{x}(b):=G_{x}(b)^{-1}.
    \end{equation}
    It verifies $\Im(F_{x}(b))\geq \Im(b)$ (see Remark 2.5 in \cite{belinschi_infinite_2012}).
\end{itemize}
They are both well defined on $\mathbb{H}^{+}(\B)$ and $F_{x}$ takes values in $\mathbb{H}^{+}(\B)$. Note that they can be defined for an arbitrary measure $\mu\in\Sigma_{0}$ thanks to Theorem \ref{caractérisation mesures algébriques}. It should also be noted that, unlike the scalar case, they do not characterize the distribution $\mu$. It is shown in the pioneering work of Voiculescu \cite{voiculescu_free_1992} that one needs more information to retrieve the distribution $\mu$, namely the \emph{fully matricial extension} of the Cauchy transform. However, we do not need this here. 

We also recall simple facts about the $\B$-valued free sum of two non commutative $\B$-valued variables $x$ and $y$ in a $C^{*}$ operator-valued non commutative probability space $(\A,\B,E)$. We denote $\A_{1}=\B\langle\X\rangle$ (resp. $\A_{2}=\B\langle\Y\rangle$) the algebra generated by $\B$ and a formal variable $\X$ (resp. $\Y$). Note that there are canonical injective maps from $\A_{1}$ and from $\A_{2}$ to $\B\langle \X,\Y\rangle$, where $\X$ and $\Y$ are algebraically free. We identify $\A_{1}$ and $\A_{2}$ to their image through these morphism. We define the map $\mu:\B\langle \X,\Y\rangle\rightarrow\B$ through the following rule. For any $n\geq 1$, let $P_{1},\cdots,P_{n}$ such that $P_{i}\in\A_{j_{i}}$ with $j_{1}\neq j_{2},\cdots,j_{n-1}\neq j_{n}$, denoting $\nu(P_{i})=\mu_{x}(P_{i})$ if $j_{i}=1$ and $\nu(P_{i})=\mu_{y}(P_{i})$ if $j_{i}=2$, we have 
\begin{equation}
    \mu\left((P_{1}-\nu(P_{1}))\cdots(P_{n}-\nu(P_{n}))\right)=0
\end{equation}

Finally for any polynomial $P\in\B\langle\mathcal{Z}\rangle$ we denote $P(\X+\Y)\in\B\langle\X,\Y\rangle$ its image through $\mathcal{Z}\mapsto \X+\Y$ and $x\boxplus_{\B} y$ as the element having distribution $\mu_{x\boxplus_{\B}y}$ defined as
\begin{equation*}
    \begin{split}
        \mu_{x\boxplus_{\B}y}:\,&\B\langle\mathcal{Z}\rangle\rightarrow \B\\
        &P\mapsto\mu(P(x+y)),
    \end{split}
\end{equation*}
with the natural evaluation map.

Hence, there exists a canonical way to build the amalgamated (over $\B$) free product between two non commutative operator-valued variable. We provide in Section 3 a more explicit construction in the case of random matrices.

\begin{remark}
    Two cases of $C^{*}$-operator-valued probability spaces will be of importance later.
    \begin{itemize}
        \item  The triplet $(\M_{N}(\C),\D_{N}(\C),\Delta)$ where $\D_{N}(\C)$ is the (commutative) \textbf{algebra of diagonal matrices} and $\Delta$ is the map defined by $\Delta(M)=(M_{i,j}\1_{i=j})_{1\leq i,j\leq N}$, is a $C^{*}$-operator-valued probability space.
        \item  Moreover, one can also take any unital $*$-algebras $\B\subset\A\subset B(H)$, such that they are both closed in the norm topology of $B(\H)$ and take any $*$-morphism from $A$ to $B$ which will automatically be completely positive (see \cite{nica_lectures_2006} Chapter 3 and \cite{conway_course_2000} Example 34.3). 
    \end{itemize}
\end{remark}

\begin{definition}
    A random matrix $M$ is said to be permutation invariant if for any $\sigma\in\mathfrak{S}_{N}$, the matrix $V_{\sigma}MV_{\sigma}^{-1}$ has the same law as $M$, where $V_{\sigma}=(\1_{i=\sigma(j)})_{1\leq i,j\leq N}$.
\end{definition}

A key result from \cite{au_large_2021} is that independent permutation invariant random matrices that are bounded in operator norm are asymptotically free over the diagonal. We state a weaker version of this result. 

\begin{theorem}[Theorem 1.2 of \cite{au_large_2021}]
    Let $A_{1}$ and $A_{2}$ be independent permutation invariant random matrices that are bounded in operator norm. Let $P_{1},\cdots,P_{n}\in \D_{N}\langle\X\rangle$ and $j_{1}\neq j_{2}\neq \cdots\neq j_{n}$, then denoting
    \begin{equation*}
        \varepsilon_{N}:=\Delta\left((P_{1}(A_{j_{1}})-\Delta(P_{1}(A_{j_{1}})))\cdots (P_{n}(A_{j_{n}})-\Delta(P_{n}(A_{j_{n}})))\right),
    \end{equation*}
    we have
    \begin{equation}
        \E\left[\frac{1}{N}\mathrm{Tr}(\varepsilon_{N}\varepsilon_{N}^{*})^{p}\right]\underset{N\rightarrow\infty}{\longrightarrow} 0.
    \end{equation}
\end{theorem}

Our goal is to compare quantitatively the law of $A+B$ to the law of the amalgamated free sum of $A$ and $B$. Heuristically, to build the amalgamated free sum, we build from the data of $A$ and $B$ a bounded operator $a+b$ on some Hilbert space which is the adjacency operator of some graph. The important feature of this graph is that its graph of colored component with respect to the families $A$ and $B$ is a tree (see Section 4.1).

    \subsection{Main theorems and sketch of the proof}
        Let $A$ be permutation invariant, we can consider $B$ to also be permutation invariant. Indeed, denoting $\Tilde{B} :=VBV^{-1}$ where $V$ is a uniform permutation matrix, we have $\E\tr G_{A+B}(zI)=\E\tr G_{A+\Tilde{B}}(zI)$. Therefore in the remainder of this paper, and in particular in the proofs of the following theorems, we consider $A$ and $B$ to be permutation invariant random matrices.

For $A$ and $B$ two random matrices we denote by $a+b$ their free sum over the diagonal, i.e. the operator constructed in Section 3. In this context, $a$ and $b$ are still free over the algebra of diagonal matrices, hence $G_{a+b}(D)$ takes values in $\D_{N}(\C)$. For simplicity, we denote $G_{X}(z)=G_{X}(zI)$ the diagonal-valued Cauchy transform of $X$ taken at a scalar matrix $zI$. For a matrix $M$, we denote by $||M||_{F}:=\sqrt{\tr(MM^{*})}$ the Frobenius norm of $F$ and $||M||_{op}$ its operator norm.

\begin{theorem}\label{thm A}
    Let $A$ and $B$ be random Hermitian matrices that satisfy Assumption \ref{H} at speed $M$ with constants $h_{1}$ and $h_{2}$ respectively and Assumption \ref{assumption frobenius} at speed $M$ with constant $C$ for both. Assume that one of them is permutation invariant. Let $0\leq c \leq \frac{1}{h_{1}+h_{2}+6}$ and $|z|^{2}> \max\{4C, 2c^{2}\}=:r_{0}^{2}$, we have
\begin{equation}
    \E\frac{1}{N}||G_{a+b}(z)-G_{A+B}(z)||_{F}^{2}\leq \frac{4}{|z|^{2}}\left(\frac{r_{0}}{|z|}\right)^{2cM}
\end{equation}
\end{theorem}
It is specified in the Assumptions that the speed $M$ cannot be greater than $\log N /\log\log N$. Therefore, the best speed one can achieve is of order $N^{\gamma/\log\log N}$.

The next theorem is more restrictive on the assumptions but provides a better speed of convergence and controls the difference of the operator-valued Cauchy transforms not only on scalar matrices.
\begin{theorem}\label{thm B}
    Let $A$ and $B$ be two Hermitian random matrices, one of them being permutation invariant. Assume that there exists $C>0$ independent of $N$ such that they are
    \begin{enumerate}
        \item Uniformly bounded in operator norm: $||A||_{op},||B||_{op}\leq C$,
        \item Sparse: $\forall 1\leq i \leq N, \sum_{j=1}^{N}\1_{A_{i,j}\neq 0}+\1_{B_{i,j}\neq 0}\leq C$.
    \end{enumerate}
    Let $\eta_{0}>||A||_{op}+||B||_{op}$ and $D\in \D_{N}(\C)$ such that $\min\{\Im(D_{i,i})\}>\eta_{0}$, and denote $d:=\min\{|D_{i,i}|,\,i\in[N]\}$, we have
    \begin{equation}
        \E\frac{1}{N}||G_{A+B}(D)-G_{a+b}(D)||_{F}^{2}\leq \frac{6}{d^{2}}N^{\frac{2\log2C/d}{1+2\log C}}.
    \end{equation}
\end{theorem}

The proof of both theorems rely on a moment method: we compare the moments of $A+B$ to the moments of the operator $a+b$ constructed in Section 3.

\section{Explicit construction of the amalgamated free sum}
    We give a constructive definition of the amalgamated free sum of $A$ and $B$ as the adjacency operator of a locally finite rooted weighted graph. This construction is inspired from \cite{timhadjelt_spectral_2025} and from the methods developed in \cite{huang_spectrum_2024,erdos_spectral_2013}.
Let $N\in \N $, for $x\in[N]:=\{1,\cdots,N\}$, we define
\begin{align*}
    V_{N}^{x}&:=\{x\}\sqcup\bigsqcup_{n=1}^{\infty} V_{N}^{x}(n,a)\sqcup  V_{N}^{x}(n,b), \text{ where}\\
    V_{N}^{x}(n,a) &:=\left\{(j_{1},\cdots,j_{n},a)\in [N]^{n}\times\{a\}| \,x\neq j_{1},\,j_{1}\neq j_{2},\cdots,\,j_{n-1}\neq j_{n}\right\},\\
    V_{N}^{x}(n,b) &:=\left\{(j_{1},\cdots,j_{n},b)\in [N]^{n}\times\{b\}| \,x\neq j_{1},\,j_{1}\neq j_{2},\cdots,\,j_{n-1}\neq j_{n}\right\}.
\end{align*}
Let $\H_{N}^{x}:=\ell^{2}(V_{N}^{x})$ be the set of sequences $(y_{v})_{v\in V_{N}^{x}}$ such that $\sum_{v\in V_{N}^{x}}|y_{v}|^{2}<+\infty$. Since $(\delta_{v})_{v\in V_{N}^{x}}$ is a basis of $\H_{N}^{x}$ (where $\delta_{v}$ denotes the sequence with 0 everywhere except at $v$), we often identify elements of this basis of $\H_{N}^{x}$ with elements of $V_{N}^{x}$. We let $V_{N}^{x}(0):=\{x\}$ for easier notations.

Note that for $x,y\in[N]$, there is a natural bijection between $V_{N}^{x}$ and $V_{N}^{y}$ where one sends a tuple $(j_{1},\dots,j_{n},a)$ to $(j_{1}',\cdots,j_{n}',a)$ with $j_{i}'=j_{i}$ if $j_{i}\notin\{x,y\}$, $j_{i}'=x$ if $j_{i}=y$ and $j'_{i}=y$ if $j_{i}=x$ and similarly for tuple ending in $b$. This bijection naturally defines an isomorphism $\Psi_{x\rightarrow y}:\H_{N}^{x}\rightarrow\H_{N}^{y}$ when acting on elements of the basis, we thus often write $\H_{N}$ instead of $\H_{N}^{x}$.

We build a family of maps, for $x\in [N]$,
\begin{equation}
    \boxplus_{\Delta}^{x}:\M_{N}(\C)\times\M_{N}(\C)\rightarrow B(\H_{N}^{x}),
\end{equation}
that verify the following property: $\forall x,y\in[N], (.\boxplus_{\Delta}^{x}.)=\Psi_{y\rightarrow x}(.\boxplus_{\Delta}^{y}.)\Psi_{x\rightarrow y}$. Hence, one just needs to construct one of those maps and we will call it simply $\boxplus_{\Delta}$.

The idea is the following, for each matrix $A$ and $B$, we construct $A\boxplus_{\Delta}B$ as follows. Viewing $A$ and $B$ as weighted graphs on the set of vertices $[N]$, we attach to each vertex $v$ of $A$, a copy of the graph $B$ by the same vertex $v$ of $B$. Then, to each new vertex created (hence vertices of one of the copies of the graph of $B$), we attach a new copy of $A$ to the corresponding vertex. Repeating this process infinitely many times, we obtain an infinite weighted locally finite graph. Its adjacency operator is in $B(\H_{N})$. 
\begin{remark}
    Starting from a different vertex $y$ leads to exactly the same graph after renaming the set of vertices as above 
\end{remark}
Figure \ref{fig: somme libre} represents the first steps of this construction when the root is the vertex $2$.

\begin{figure}[h]
    \centering
    % Ligne du haut : deux sous-figures côte à côte
    \begin{subfigure}[t]{0.3\textwidth}
        \centering
        \includegraphics[width=\textwidth]{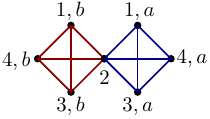}
        \caption{First step}
    \end{subfigure}
    \hfill
    \begin{subfigure}[t]{0.4\textwidth}
        \centering
        \includegraphics[width=\textwidth]{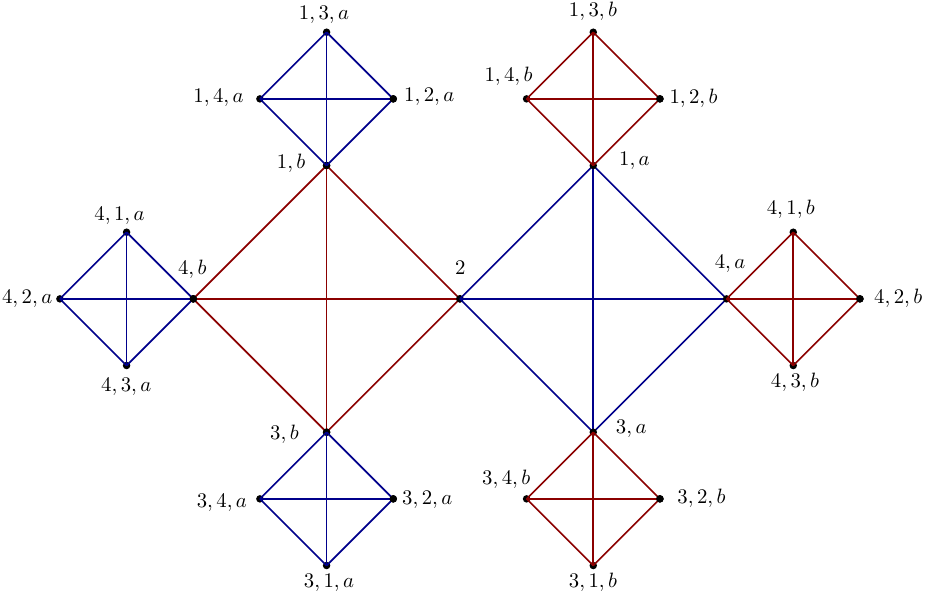}
        \caption{Second step}
    \end{subfigure}
    
    % Ligne du bas : une seule sous-figure
    \vskip\baselineskip
    \begin{subfigure}[t]{0.5\textwidth}
        \centering
        \includegraphics[width=\textwidth]{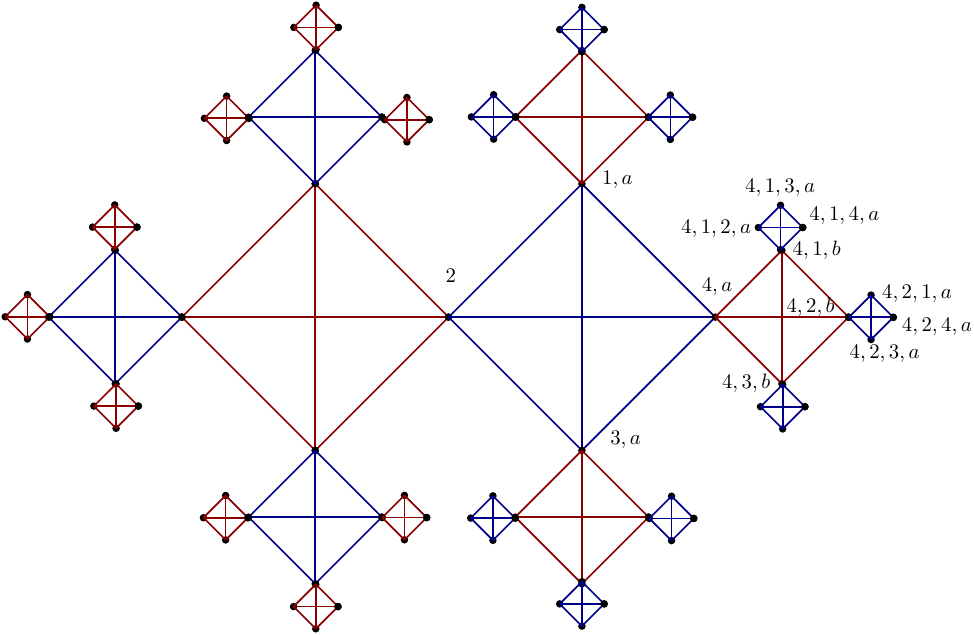}
        \caption{Third step}
    \end{subfigure}
    
    \caption{Construction of $A\boxplus_{\Delta}B$ as an iterating process. The blue (resp. red) edges are labeled $a$ (resp. $b$). }
    \label{fig: somme libre}
\end{figure}

It should be noted that this graph has a tree-like structure when looking at the components that come from $A$ and the components that come from $B$. Formally, its graph of colored components with respect to the families $A$ and $B$ is an infinite tree (see Section 4).

We now explicitly construct operators $a,b\in B(\H_{N}^{x})$ associated to $A$ and $B$ such that $A\boxplus_{\Delta}B=a+b$. Then, we show that they are free over the algebra of diagonal operators and even above the algebra of diagonal matrices in a sense to be made precise below. Therefore, we define the amalgamated free sum of $A$ and $B$ as the sum of $a$ and $b$ (where the sum is now taken with respect to the natural sum operation on bounded linear operators).

In order to facilitate the definition of $a$ and $b$, we use the  convention $V_{N}(0,a)=V_{N}(0,b)=V_{N}(0)=\{x\}$ and when $n=0$, $(j_{1},\dots,j_{n},a)=(j_{1},\dots,j_{n},b)=x$. With this, one has the following decompositions of the set $V_{N}^{x}$ as
\begin{equation}\label{decomposition sommets}
    \begin{split}
        V_{N}^{x}&=\bigsqcup_{n=0}^{\infty}\bigsqcup_{v\in V_{N}(n,b)}R_{a}(v)=\bigsqcup_{n=0}^{\infty}\bigsqcup_{w\in V_{N}(n,a)}R_{b}(w),
    \end{split}
\end{equation}
where, for $n\geq0$, $v=(j_{1},\cdots,j_{n},b)\in V_{N}(n,b)$ and $w=(j_{1},\cdots,j_{n},a)\in V_{N}(n,a)$, we denote
\begin{equation}
    \begin{split}
        R_{a}(v)&:=\{v\}\sqcup\{(j_{1},\cdots,j_{n},j_{n+1},a),\,j_{n+1}\in[N]\setminus\{j_{n}\}\},\\
        R_{b}(w)&:=\{w\}\sqcup\{(j_{1},\cdots,j_{n},j_{n+1},b),\,j_{n+1}\in[N]\setminus\{j_{n}\}\}.
    \end{split}
\end{equation}
Looking at Figure \ref{fig: somme libre}, it is easy to understand this decomposition. One groups vertices that belong to a same colored component: blue for $R_{a}$ and red for $R_{b}$. We call the sets $R_{a}(v),R_{b}(w)$ components in $a,b$ respectively.

With this decomposition, we define the operators $a$ and $b$ as follows. 

Let $n\geq 0$ and $v=(j_{1},\cdots,j_{n},b)\in V_{N}(n,b)$. Let $v_{1},v_{2}\in R_{a}(v)$, we denote $j(v_{1})$ and $j(v_{2})$ the last index of $v_{1}$ and $v_{2}$ respectively (it might be $j_{n}$ or some $j_{n+1}\neq j_{n}$) and we set $\langle a\delta_{v_{1}},\delta_{v_{2}}\rangle = A(j(v_{2}),j(v_{1}))$. 

Similarly, let $n\geq0$ and $w=(j_{1},\cdots,j_{n},a)\in V_{N}(n,a)$. Let $w_{1},w_{2}\in R_{b}(w)$, we denote $j(w_{1})$ and $j(w_{2})$ the last index of $w_{1}$ and $w_{2}$ respectively (it might be $j_{n}$ or some $j_{n+1}\neq j_{n}$) and we set $\langle b\delta_{w_{1}},\delta_{w_{2}}\rangle = B(j(w_{2}),j(w_{1}))$.

Let $n\geq 0$, let $v\in V_{N}(n,b)$ for any $v_{1}\in R_{a}(v)$, we define 
\begin{equation}
    a:\delta_{v_{1}}\mapsto\sum_{v_{2}\in R_{a}(v)}\langle a\delta_{v_{1}},\delta_{v_{2}}\rangle\delta_{v_{2}}.
\end{equation}
Similarly, let $n\geq 0$, let $w\in V_{N}(n,a)$ for any $w_{1}\in R_{b}(w)$, we define 
\begin{equation}
    b:\delta_{w_{1}}\mapsto\sum_{w_{2}\in R_{b}(v)}\langle b\delta_{w_{1}},\delta_{w_{2}}\rangle\delta_{w_{2}}.
\end{equation} 

Notice that $a$ and $b$ respect the decomposition given by \eqref{decomposition sommets} in the sense that any vertex $v$ is sent by $a$, resp. $b$, to vertices in its components in $a$ resp. $b$.

The matrices $A$ and $B$ being Hermitian, it is easy to see that $a$ and $b$ are bounded self-adjoint operators with same spectral norm than $A$ and $B$ respectively.

Recall that for $T\in B(\H_{N}^{x})$, we define the diagonal of $T$ as 
    \begin{equation}
        \Delta(T):\delta_{v}\mapsto \langle T(\delta_{v}),\delta_{v}\rangle \delta_{v},
    \end{equation}
and note that $\Delta(B(\H_{N}^{x}))=:\D$ is a closed subalgebra of $B(\H_{N}^{x})$ for the norm topology.

Let $\A=\mathrm{cl\ span}\{a,b,\D\}$ where $\mathrm{cl}$ is the closure in $B(\H_{N}^{x})$ for the norm topology. The triplet  $(\A,\D, \Delta)$ is a $C^{*}$-operator-valued non commutative probability space.

\begin{proposition}
    The operators $a$ and $b$ are $\Delta$-free in $(\A,\D, \Delta)$.
\end{proposition}
\begin{proof}
    We denote $X_{1}$ (resp. $X_{2}$) the subalgebra of $\A$ generated by $a$ and $\D$ (resp. $b$ and  $\D$). Note that this is not only the vector space spanned by monomials of the form $(ad)^{k}$ for $d\in \D$ but since $a$ and $d$ do not commute in general, it is spanned by monomials of the form $d_{0}ad_{1}\cdots ad_{k}$, for $d_{0},d_{1},\cdots,d_{k}\in\D$.

    Let $m\geq 2$ and $x_{1},\cdots,x_{m}$ be such that $x_{2i-1}\in X_{1}$ and $x_{2i}\in X_{2}$ for all $1\leq i\leq \lfloor m/2\rfloor $ and such that for all $1\leq i \leq m$, $\Delta(x_{i})=0$. We want to show that $\Delta(x_{1}\cdots x_{m})=0$. Let $v\in V_{N}^{x}$, we want to compute $\langle x_{1}\cdots x_{m}(\delta_{v}),\delta_{v}\rangle$. First assume that $x_{m}\in X_{1}$ and $v\in V_{N}(n,b)$ for some $n\geq 0$. Since $x_{m}$ is centered, we have
    \begin{equation}
        x_{m}(\delta_{v})=\sum_{v'\in R_{a}(v)\setminus\{v\}}\langle x_{m}\delta_{v},\delta_{v'}\rangle\delta_{v'}.
    \end{equation}
    Since $R_{a}(v)\setminus\{v\}\subset V_{N}(n+1,a)$, by induction, we have that $x_{1}\cdots x_{m}(\delta_{v})\subset V_{N}(n+m,a)\cup V_{N}(n+m,b)$, hence $\Delta(x_{1}\cdots x_{m})=0$. The same reasoning applies when $x_{m}\in X_{2}$ and $v\in V_{N}(n,a)$ for some $n\geq 0$.

    Let us now assume that  $x_{m}\in X_{1}$ and $v=(j_{1},\cdots,j_{n},a)\in V_{N}(n,a)$ for some $n\geq 0$. To simplify notations, for any $k\geq 0$, we define by convention $V_{N}(-k,a)=V_{N}(k,a)$ and similarly for $b$. The vertex $v_{m-1}=(j_{1},\cdots,j_{n-1},b)\in V_{N}(n-1,b)$ is such that $v\in R_{a}(v_{m-1})$, thus
    \begin{equation*}
        x_{m}(\delta_{v})=\langle x_{m}\delta_{v},\delta_{v_{m-1}}\rangle\delta_{v_{m-1}}+\sum_{w\in R_{a}(v_{m-1})\setminus\{v,v_{m-1}\}}\langle x_{m}\delta_{v},\delta_{w}\rangle\delta_{w}.
    \end{equation*}
    For any $w\in R_{a}(v_{m-1})\setminus\{v,v_{m-1}\}$,  $w=(j_{1},\cdots,j_{m-1},j,a)$ for some $j \neq j_{m}$. Hence $x_{1}\cdots x_{m-1}(\delta_{w})$ is supported by vertices of the form $(j_{1},\dots,j_{n-1},j,i_{1},\cdots,i_{m-1},\varepsilon)$ for $\varepsilon\in\{a,b\}$ which can never be equal to $v$ since $j\neq j_{m}$. We can therefore disregard the contribution of these terms.

    We repeat this process for $v_{m-1}$: the vertex $v_{m-2}=(j_{1},\dots,j_{n-2},a)\in V_{N}(n-2,a)$ is such that $v_{m-1}\in R_{b}(v_{m-2})$. Hence
    \begin{equation*}
        x_{m-1}(\delta_{v_{m-1}})=\langle x_{m-1}\delta_{v_{m-1}},\delta_{v_{m-2}}\rangle\delta_{v_{m-2}}+\sum_{w\in R_{b}(v_{m-2})\setminus\{v_{m-1},v_{m-2}\}}\langle x_{n}\delta_{v},\delta_{w}\rangle\delta_{w}.
    \end{equation*}
    Any $w\in R_{b}(v_{m-2})\setminus\{v_{m-1},v_{m-2}\}$ is of the form $w=(j_{1},\cdots, j_{n-2},j,b)$ for some $j\neq j_{n-1}$, hence $x_{1}\cdots x_{m-2}(\delta_{w}) $ will be supported by vertices of the form $(j_{1},\cdots,j_{n-2},j,i_{1},\cdots,i_{m-2},\varepsilon)$ for $\varepsilon\in\{a,b\}$. Again, we can disregard the contribution of those terms.

    If $m\leq n$, this reasoning can be applied $m$ times and at each step, $v_{k}\in V_{N}(n-(m-k),\varepsilon)$ for $\varepsilon\in\{a,b\}$, therefore, the product $x_{1}\cdots x_{m}$ is centered. 

    Otherwise, if $m>n $, we apply this reasoning $n$ times until $v_{m-n}=x$. We next apply $x_{m-n}$. Hence $x_{1}\cdots x_{m-n}(\delta_{x})$ is supported on vertices of the form $(i_{1},\dots,i_{m-n},\varepsilon)$  for $\varepsilon\in \{a,b\}$. Hence the only non-trivial case happens when $m-n=n$ which implies that $m$ is even. If $n$ is even, $x_{m-n}\in X_{1}$ otherwise $x_{m-n}\in X_{2}$. Furthermore, if $n$ is even, this implies on one hand $x_{m-n}\in X_{1}$ and on the other hand that $m-n$ is even. Therefore one must have $\varepsilon =b$: indeed we apply the reasoning of the first scenario to see that one takes an even number of steps starting with a step labeled $a$ which will end in a step labeled $b$. Similarly, if $n$ is odd, this implies that $x_{m-n}\in X_{2}$ and $m-n$ is odd which in turns implies $\varepsilon=b$. In both cases the scalar product of $x_{1}\cdots x_{m-n}(\delta_{x})$ with $\delta_{v}$ will lead to a zero contribution.
\end{proof}

\begin{remark}
    Writing the proof and understanding it as it is is quite tedious. However, one should draw the graph associated to $a$ and the one associated to $b$ and see that, alternating between edges from $a$ and edges from $b$ while never staying at the same vertex (the $x_{i}$'s are centered), one can never return to the starting vertex.
\end{remark}

\begin{remark}
    There is a natural embedding of $\D_{N}(\C)$ into $\D$: for $D\in\D_{N}(\C)$, we still denote the associated diagonal operator $D \in\D$ defined by $ D(\delta_{v})=D_{j(v),j(v)}\delta_{v}$, for any $v\in V_{N}$, where $j(v)$ denotes the last index of $v$. We therefore identify $\D_{N}(\C)$ with the image of $\D_{N}(\C)$ through this embedding.
\end{remark}

Defining $E_{\D_{N}(\C)}:\A\rightarrow \D_{N}(\C),\;T\mapsto (\langle T\delta_{(i,a)},\delta_{(i,a)}\rangle)_{1\leq i \leq N}$, it is not hard to see that $(\A,\D_{N}(\C),E_{\D_{N}(\C)})$ is a $C^{*}$-operator-valued probability space

\begin{corollary}\label{free over diagonal matrices}
    The operators $a$ and $b$ are free in the $C^{*}$-operator-valued probability space $(\A,\D_{N}(\C),E_{\D_{N}(\C)})$.
\end{corollary}
\begin{proof}
    The algebra $\D_{N}(\C)$ is a subalgebra of $\D(\H_{N})$ and $E_{\D_{N}(\C)}=E_{\D_{N}(\C)}\circ\Delta$. Moreover, for any diagonal matrices $D_{0},\cdots,D_{n}$, the moment $\Delta(D_{0}aD_{1}\cdots aD_{n})$ is actually in $\D_{N}(\C)$ and the same goes when replacing $a$ by $b$. Indeed, the definition of $a$ and $b$ only depends on the last index of the vertex $v$. Hence, the $ \Delta$-moments of $a$ and $b$ are in $\D_{N}(\C)$, therefore, they are equal to their $E_{\D_{N}(\C)}$-moments. Proposition 16, Chapter 9 of \cite{mingo_free_2017} concludes the proof.
\end{proof}

Moreover, let $\varphi:\A\rightarrow\C,\,T\mapsto \frac{1}{N}\sum_{v\in R_{a}(x)}\langle T\delta_{v},\delta_{v}\rangle$, be a unital positive $*$-morphism so that it endows $\A$, $\D$ and $\D_{N}(\C)$ with a structure of $C^{*}$-probability space. Note that in the literature, a $C^{*}$-operator-valued probability space is sometimes required to be equipped with such a positive linear functional $\varphi:\A\rightarrow\C$.

A key result that one may find in \cite{nica_lectures_2006} is that for any self-adjoint element $x$ in a $C^{*}$-probability space $(\A_{0},\varphi_{0})$, there exists a compactly supported real measure $\mu_{x}$ (with support included in the spectrum of $x$) such that for every continuous function $f$ on the spectrum of $x$, $\int f\dd\mu_{x}=\varphi_{0}(f(x))$, where on the right-hand side $f(x)$ is defined by functional calculus and on the left-hand side $\mu_{x}$ is viewed as a measure on the spectrum of $x$. This can then be applied to $(\A,\varphi)$.

\section{Traffic notions and assumptions}
    The last ingredients we need in order to prove Theorems \ref{thm A} and \ref{thm B} come from traffic probability. It has been introduced by Male in \cite{male_traffic_2020}, for the sake of completeness, we restate some of the relevant definitions and results required in this work. 
    \subsection{Traffic notions}

\begin{definition}[Graph of colored components]\label{GCC}
    Let $H$ be a graph with edges labeled by a family $\{x_{i}\}_{i\in I}$ (each $x_{i}$ may itself be a family of several labels as long as $x_{i}\cap x_{j}=\varnothing$ for $i\neq j$). 

    A \emph{colored component} $T$ of $H$ is a maximal subgraph with at least one edge whose edges are labeled only with labels from a single family. We denote by $\mathcal{CC}(H)$ the set of colored components of $H$.
    
    The \emph{graph of colored components} of $H$ with respect to the families $\{x_{i}\}_{i\in I}$ denoted $\mathcal{GCC}(H)$ is the following bipartite graph:
    \begin{itemize}
        \item The first kind of vertices are the colored components $T_{1},\cdots,T_{K}$ of $H$.
        \item The second kind of vertices are vertices $v_{1},\cdots,v_{L}$ of $H$ that belong to at least two colored components.
        \item There is an edge between $v_{i}$ and $T_{j}$ if $v_{i}\in T_{j}$ for all $1\leq i\leq L$, $1\leq j \leq K$.
    \end{itemize}
    
\end{definition}

For a given connected simple graph $H=(V,E)$, we denote 
\begin{equation}\label{Euler}
    \eta(H):=|V|-|E|-1.
\end{equation}

Note that $\eta(H)\leq0$ with equality if and only if $H$ is a tree (see \cite{guionnet_large_2009} Lemma 1.1).

\begin{remark}\label{eta GCC}
    We will often talk about the graph of colored components of $A+B$, implicitly stating that it is with respect of the family of edges coming from $A$ and the family of edges coming from $B$. In that case, we denote $\mathcal{V}_{cc}$ (for colored component) the first kind of vertices of the graph of colored components of $H$ and $\mathcal{V}_{co}$ (for  connectors) the second kind of vertices. Denoting $\mathcal{E}$ the set of edges of $\GCC(H)$, notice that $|\mathcal{E}|=2|\mathcal{V}_{co}|$, hence $\eta(\GCC(H))=|\mathcal{V}_{cc}|-|\mathcal{V}_{co}|-1$.
\end{remark}

\begin{definition}
        Let $J$ be a set of indices and consider two families of formal variables $\mathbf{x}=(x_{j})_{j\in J}$ and $\mathbf{x^{*}}=(x_{j}^{*})_{j\in J}$.
	    \begin{enumerate}
	        \item A \emph{*-test graph} $T = (G, \gamma, \epsilon )$ in the variables $\mathbf{x}$ is a finite connected multi-digraph $G = (V , E, \mathrm{src}, \mathrm{tar})$ together with edge labels $\epsilon: E \rightarrow \{1,*\}$ and $\gamma : E \rightarrow J$. One can see the maps $\gamma$ and $\epsilon$ as indicating that an edge $e\in E$ is labeled $x_{\gamma(e)}^{\epsilon(e)}$. The maps $\mathrm{src}$, $\mathrm{tar} : E \rightarrow V$ specify the source, $\mathrm{src}(e)$ and target, $\mathrm{tar}(e)$ of each edge $e \in E$. 
            \item A \emph{*-graph monomial} $g=(T,v_{in},v_{out})$ is a *-test graph $T$ together with the data of two vertices  $v_{in}$ and $v_{out}$ in $V$. We refer to the roots $(v_{in}, v_{out}) \in V^{2}$ as the input and the output, respectively, though they need not be distinct. We denote by $\mathcal{G}\langle\mathbf{x},\mathbf{x^*}\rangle$ the set of all such *-graph monomials which can be extended to $\C\mathcal{G}\langle\mathbf{x},\mathbf{x^*}\rangle$ the space of finite linear complex combinations of *-graph monomials in variable $\mathbf{x}$ (graphs are considered up to isomorphisms of graphs preserving labeled and in/outputs).
	    \end{enumerate}
	\end{definition}
    
    \begin{example}\label{traffic matrices}
        For a family $\mathbf{A}_{N}=(A_{N,j})_{j\in J}$ of random matrices, we define the evaluation of a *-graph monomial $g$ in the family $\mathbf{A}_{N}$ via the formula,
    \begin{equation}
        g(\mathbf{A}_{N})(i,j):=\sum_{\substack{\phi:V\rightarrow[N] \\ \phi(v_{out})=i,\phi(v_{in})=j}}\prod_{e\in E}A^{\epsilon(e)}_{N,\gamma(e)}(\phi(\mathrm{tar}(e)),\phi(\mathrm{src}(e))).
    \end{equation}
    For convenience, we will often denote $\phi(e)$ for $(\phi(\mathrm{tar}(e)),\phi(\mathrm{src}(e)))$.
    When one studies random matrices in the framework of free probability, one is interested in the expectation of the normalized trace of any polynomials evaluated in a family of random matrix. In the context of this section, the \emph{traffic distribution} of a family of random matrices is the data of the expectation of the normalized trace of any *-graph monomial evaluated in this family.
    \end{example}

    \begin{example}\label{traffic operators}
    Let $\mathcal{G}=(\mathcal{V},\mathcal{E})$ be a locally finite rooted graph (possibly dependent on $N$). Let $\mathbf{a}_{N}=(a_{N,j})_{j\in J}$ be a family of random  operators on $B(L^{2}(\mathcal{V}))$ such that for all $j\in J$, $\langle a_{N,j}\delta_{v},\delta_{w}\rangle=0$ if $(v,w)$ is not an edge of $\mathcal{G}$. We define the evaluation of a *-graph monomial $g$ in the family $\mathbf{a}_{N}$ via the formula,
    \begin{equation}
        \langle g(\mathbf{a}_{N})\delta_{v},\delta_{w}\rangle:=\sum_{\substack{\phi:V\rightarrow\mathcal{V} \\\text{graph morphism},\\ \phi(v_{out})=v,\phi(v_{in})=w}}\prod_{e\in E}\langle a^{\epsilon(e)}_{N,\gamma(e)}\delta_{\phi(\mathrm{tar}(e))},\delta_{\phi(\mathrm{src}(e))}\rangle.
    \end{equation}
    For convenience, we will often denote $a(\phi(e))$ for $\langle a\delta_{\phi(\mathrm{tar}(e))},\delta_{\phi(\mathrm{src}(e))}\rangle$. Note that $g(\mathbf{a}_{N})$ does not necessarily verify the condition that $\langle g(\mathbf{a}_{N})\delta_{v},\delta_{w}\rangle=0$ if $(v,w)\notin \mathcal{E}$ however, it can be seen as an adjacency operator of some weighted locally finite graph. Note that this coincides with the previous definition when $\mathcal{G}$ is the complete graph on $N$ vertices. 
    \end{example}

    The following definitions are usually defined for test-graphs. We define their counterpart for (rooted) graph monomials and we keep the same names.

    \begin{definition}\label{def: traffic distribution}
        Let $i\in[N]$, the $i$-rooted \emph{traffic distribution} of a family $\mathbf{A}_{N}$ of random matrices is the map
        \begin{equation}
            \Phi^{i}_{\mathbf{A}_{N}}:g\in\C\mathcal{G}\langle\mathbf{x},\mathbf{x^*}\rangle\rightarrow\E\left[g(\mathbf{A}_{N})_{i,i}\right].
        \end{equation}

        Let $\mathbf{a}_{N}$ be a family of random operators acting on $L^{2}(\mathcal{V})$ where $\mathcal{G}=(\mathcal{V},\mathcal{E},\rho)$ is a locally finite $\rho$-rooted graph dependent on $N$. The $\rho$-rooted traffic distribution of $\mathbf{a}_{N}$ is
        \begin{equation}
            \Phi^{\rho}_{\mathbf{a}_{N}}:g\in\C\mathcal{G}\langle\mathbf{x},\mathbf{x^*}\rangle\rightarrow\E\left[\langle g(\mathbf{a}_{N})\delta_{\rho},\delta_{\rho}\rangle\right].
        \end{equation}
        With a slight abuse of notation, we also denote $\tau_{i}(g(\mathbf{A}_{N}))=\Phi^{i}_{\mathbf{A}_{N}}(g)$ and similarly for operators.
    \end{definition}

    Note that these two definitions coincide when $\G$ is the complete graph on $N$ vertices. It is natural in a traffic probability framework to decompose the sums in Examples \ref{traffic matrices} and \ref{traffic operators} depending on the default of injectivity of $\phi$. More precisely,
    \begin{align*}
        g(\mathbf{A}_{N})(i,j)&=\sum_{\pi\in P(V)}\sum_{\substack{\phi:V^{\pi}\rightarrow[N],\\\phi\text{ injective}, \\ \phi(v_{out})=i,\phi(v_{in})=j}}\prod_{e\in E^{\pi}}A^{\epsilon(e)}_{N,\gamma(e)}(\phi(\mathrm{tar}(e)),\phi(\mathrm{src}(e))),\\
        \langle g(\mathbf{a}_{N})\delta_{v},\delta_{w}\rangle&=\sum_{\pi\in P(V)}\sum_{\substack{\phi:V^{\pi}\rightarrow\mathcal{V} \\\text{injective graph morphism},\\ \phi(v_{out})=v,\phi(v_{in})=w}}\prod_{e\in E^{\pi}}\langle a^{\epsilon(e)}_{N,\gamma(e)}\delta_{\phi(\mathrm{tar}(e))},\delta_{\phi(\mathrm{src}(e))}\rangle,
    \end{align*}
    where a partition $\pi$ of the vertices of the graph monomial $g$ induces a new graph monomial $g^{\pi}=(V^{\pi},E^{\pi},\gamma,\varepsilon,v_{in},v_{out})$ obtained from $g$ after identifying vertices in a same block of $\pi$. This motivates the following definition
    \begin{definition}
    The $i$-rooted injective trace of $g$ in the matrices $\mathbf{A}_{N}$, and the $\rho$-rooted injective trace of $g$ in the operators $\mathbf{a}_{N}$ is 
        \begin{align*}
            \tr^{0}_{i}(g(\mathbf{A}_{N}))&=\sum_{\substack{\phi:V\rightarrow[N],\\\phi\text{ injective}, \\ \phi(v_{out})=\phi(v_{in})=i}}\prod_{e\in E}A^{\epsilon(e)}_{N,\gamma(e)}(\phi(\mathrm{tar}(e)),\phi(\mathrm{src}(e))),\\
            \tr^{0}_{\rho}(g(\mathbf{a}_{N}))&=\sum_{\substack{\phi:V\rightarrow\mathcal{V} \\\text{injective graph morphism},\\ \phi(v_{out})=\phi(v_{in})=\rho}}\prod_{e\in E}\langle a^{\epsilon(e)}_{N,\gamma(e)}\delta_{\phi(\mathrm{tar}(e))},\delta_{\phi(\mathrm{src}(e))}\rangle,
        \end{align*}
        so that
        \begin{align*}
            \tau_{i}(g(\mathbf{A}_{N}))&=\sum_{\pi\in P(V)}\E\tr^{0}_{i}(g^{\pi}(\mathbf{A}_{N}))=:\sum_{\pi\in P(V)}\tau^{0}_{i}(g^{\pi}(\mathbf{A}_{N})),\\
            \tau_{\rho}(g(\mathbf{a}_{N}))&=\sum_{\pi\in P(V)}\E\tr^{0}_{\rho}(g^{\pi}(\mathbf{A}_{N}))=:\sum_{\pi\in P(V)}\tau^{0}_{\rho}(g^{\pi}(\mathbf{A}_{N})).
        \end{align*}
        Note that if $\pi$ does not identify $v_{in}$ and $v_{out}$, the injective trace is zero.
    \end{definition}

    \begin{remark}
        In \cite{male_traffic_2020} it is shown that whenever the families $\mathbf{A}_{N}^{(1)},\cdots,\mathbf{A}_{N}^{(L)}$ are permutation invariant, converge in traffic distribution and satisfy some factorization property, then they are asymptotically traffic independent. Moreover, it is shown in \cite{au_large_2021} that, provided the same condition holds, they are asymptotically free over the diagonal. 
    \end{remark}

In the following, we still call $T_{N}$ a test graph where $T_{N}:=(V_{N},E_{N},\gamma,\epsilon)$ is no longer fixed with respect to $N$ but is allowed to have a number of vertices and edges dependent of $N$. Furthermore, when specified, we also allow the test-graph to have several connected components, however one should be careful when dealing with this case since the injective trace does not behave well with multiple components. Again, when unspecified, a generic test-graph is always finite and have a single connected component.

    \subsection{Assumption on the injective trace}
    We provide in this subsection one of the assumptions of Theorem \ref{thm A}. Heuristically, this is a boundedness condition for large "moments" but in traffic framework, i.e. for large test graphs. We then provide several random matrix models that satisfy this condition.

\begin{assumption}\label{H}
    Let $M=M(N)$ be a sequence that goes to infinity with $N$ and that is smaller than $\log N / \log \log N$. The sequence $X_{N}$ of random matrices satisfies Assumption \ref{H} at speed $M$ if there exists some positive constant $h$ such that for any positive constant $c$ and for any sequence of test-graphs $T_{N}:=(V_{N},E_{N},\epsilon)$ such that $T_{N}$ has $K(N)$ connected components: $T_{1,N},\cdots,T_{K(N),N}$ $(T_{i,N}:=(V_{i,N},E_{i,N}))$ verifying
    \begin{equation}
        \sum_{k=1}^{K(N)}|V_{k,N}|\leq c M,\quad \sum_{k=1}^{K(N)}|E_{k,N}|\leq cM,
    \end{equation}
    we have
    \begin{equation}\label{equation assumption bound injective trace}
        \E\left[\frac{1}{N^{K(N)}}\mathrm{Tr}^{0}\left(T_{N}\left(X_{N}\right)\right)\right]=O\left(N^{c\cdot h}\right),
    \end{equation}
    for some constant $h$.
\end{assumption}

First note that $M\leq \log N$, so we often show that the quantity in the left-hand side of \eqref{equation assumption bound injective trace} is bounded by $C^{M}$ for some constant $C$ which yields equation \eqref{equation assumption bound injective trace}. 
We check that several model of random matrices verify these assumptions.

\paragraph{Diluted matrices with bounded entries.}
Let $X_{N}$ be a (possibly random) matrix of size $N$ such that almost surely, the number of non zero entries of $X$ per row and per column is bounded by $C=C(N)$. Moreover, assume also that the absolute value of the entries are bounded by $C$. The speed $C$ must be such that $\log C / \log N\rightarrow 0 $ as $N\rightarrow\infty$.  We first use the bound
\begin{equation*}
    \left|\E\left[\frac{1}{N^{K(N)}}\mathrm{Tr}^{0}\left(T_{N}(X_{N})\right)\right]\right|\leq\E\left[\frac{1}{N^{K(N)}}\mathrm{Tr}^{0}\left(T_{N}(|X_{N}|)\right)\right].
\end{equation*}

Now choose an arbitrary vertex in each connected component, say $v_{1},\cdots,v_{K(N)}\in V_{1,N}\times\cdots\times V_{K(N),N}$, and sum over all possible values they can take. Since we sum over injective vertex assignments, the other vertices must take a different value. Denoting $V'_{N}:=V_{N}\setminus\{v_{1},\cdots,v_{K(N)}\}$ and $[N]':=\{1,\cdots,N\}\setminus\{i_{1},\cdots,i_{K(N)}\}$ for short, we get
\begin{equation*}
    \frac{1}{N^{K(N)}}\mathrm{Tr}^{0}(T(|X_{N}|))=\frac{1}{N^{K(N)}}\sum_{\substack{i_{1},\cdots,i_{K(N)}=1,\\\text{pairwise distinct}}}^{N}\sum_{\substack{\phi:V'_{N}\rightarrow[N]',\\\phi\text{ injective}}}\prod_{e\in E_{{N}}}|X(\phi(e))|,
\end{equation*}
where $\phi$ is extended on $V_{N}$ by $\phi(v_{k})=i_{k}$ for $1\leq k \leq K(N)$.
For every $1\leq k\leq K(N)$, we do the following reasoning. After having chosen the value of $v_{k}$, since $X_{N}$ is sparse, for each vertex $w$ at distance 1 of $v_{k}$ there is at most $C$ possible values for $\phi(w)$ that lead to a non zero contribution. The same goes for vertices at distance 2 from $v_{k}$: after having chosen the values of vertices at distance $\leq2$ from $v_{k}$, there are at most $C$ possible values that lead to a non zero contribution for each new vertex at distance $2$. Hence, after having chosen the value of the vertex $v_{k}$, we have at most $C^{|V_{k,N}|}$ choices for the remaining vertices in the connected component of $v_{k}$. Furthermore, since the entries are bounded, each contribution is less than $C^{|E_{k,N}|}$. We obtain the bound
\begin{align*}
    \left|\E\left[\frac{1}{N^{K(N)}}\mathrm{Tr}^{0}\left(T_{N}(X_{N})\right)\right]\right|&\leq\frac{N(N-1)\cdots (N-K(N)+1)}{N^{K(N)}}\prod_{k=1}^{K(N)}C^{|V_{k,N}|+|E_{k,N}|},\\
    &\leq C^{2cM}.
\end{align*}
Choosing $M$ such that $M\log C =\log N $ is the best speed we can achieve provided $M\leq \log N/\log\log N$ and $\log N / \log C \rightarrow \infty$. Otherwise, one may just take $M = \log N / \log \log N$.
Note that random permutation matrices, adjacency matrices of uniform $C$-regular graphs are special cases of such sparse bounded matrices.

\paragraph{Erdos-Renyi matrices.}

Let $B_{N}$ be the adjacency matrix of an Erdos-Renyi random graph: a graph with set of vertices $[N]$ and where each (undirected) edge $e=\{i,j\}$ with $i\neq j$ is present with probability $p=p(N)$ for some $0<p<1$ such that $p\rightarrow 0$ as $N\rightarrow\infty$ and $Np\rightarrow\infty$. Note that since the edges are undirected $X_{N}$ is symmetric with 0 on the diagonal and i.i.d. sub-diagonal entries following a Bernouilli law of parameter $p$. Let $X_{N}=\frac{B_{N}-p(J-I)}{\sqrt{Np(1-p)}}$, be the centered, normalized version of $B_{N}$ where $J$ is the full 1 matrix and $I$ the identity.
Note that if the test graph $T_{N}$ has a loop, its injective trace is zero since $X_{N}$ only has 0's on its diagonal. Let us assume that $T_{N}$ does not have any loops and let us denote $\overline{E}$ the set of edges once we forget the multiplicity and the orientation of the edges of $T_{N}$. We denote $m(e)$ the multiplicity of the edge $e$ when forgetting about the orientation of the edges. Using the independence of the entries, we can explicitly compute
\begin{align*}
    \E&\left[\frac{1}{N^{K(N)}}\mathrm{Tr}^{0}\left(T_{N}(X_{N})\right)\right]=\frac{1}{N^{K(N)}}\sum_{\substack{\phi:V\rightarrow[N],\\\phi\text{ injective}}}\E\left[\prod_{\overline{e}\in\overline{E}}X(\phi(e))^{m(e)}\right]\\
    &=\frac{1}{N^{K(N)}}\#\left\{\substack{\phi:V\rightarrow[N],\\\phi\text{ injective}}\right\}\prod_{\bar e\in\bar E}\left(p\left(\frac{1-p}{\sqrt{Np(1-p)}}\right)^{m(\bar e)}+ (1-p)\left(\frac{-p}{\sqrt{Np(1-p)}}\right)^{m(\bar e)}\right),\\
    &=N^{|V|-K(N)-|\overline{E}|}(1+o(1))\prod_{\bar e\in\bar E}\left((1-p)\left(\frac{1-p}{Np}\right)^{\frac{m(\bar e)}{2}-1}+(-1)^{m(\bar e)}p\left(\frac{p}{N(1-p)}\right)^{\frac{m(\bar e)}{2}-1}\right),
\end{align*}
where we used Proposition \ref{Stirling} for the last line. If there is some edge of multiplicity exactly one, since the entries are centered, the above quantity is 0. Otherwise the second term in each term of the product is negligible and we obtain
\begin{align*}
    \E\left[\frac{1}{N^{K(N)}}\mathrm{Tr}^{0}\left(T_{N}(X_{N})\right)\right]&=N^{|V|-K(N)-|\overline{E}|}(1+o(1))\prod_{\bar e\in\bar E}\left(\frac{1}{Np}\right)^{m(\bar e)/2-1},\\
    &=N^{|V|-K(N)-|\overline{E}|}(Np)^{|\bar E|-|E|/2}(1+o(1)).
\end{align*}

By Euler's identity, we have $|V|-K(N)-|\overline{E}|\leq 0$ with equality if and only if each connected component is a tree when forgetting about the multiplicity and the orientation of the edges. Since the edges all have at least multiplicity 2, $|\bar E|-|E|/2\leq 0$ with equality if and only if each connected component of $T_{N}$ is a double tree. Finally, we get the bound
\begin{equation}
    \E\left[\frac{1}{N^{K(N)}}\mathrm{Tr}^{0}\left(T_{N}(X_{N})\right)\right]=O(1).
\end{equation}

This assumption is also satisfied for $Np\rightarrow c>0$, however, this model does not verify Assumption \ref{assumption frobenius}.

\paragraph{ Diluted Wigner matrices with moment conditions.} 
Erdös-Renyi matrices as above are a special case of diluted Wigner matrices in the following sense. 

Let $X=(x^{N}_{i,j})1\leq i,j\leq N$ be a Hermitian matrix such that the $(x^{N}_{i,j})1\leq i\leq j \leq N$ are centered, independent random variables that have moments up to order $cM$ (where $M$ goes to infinity and is smaller than $\log N / \log\log N$) and such that there exists some constants $h,\varepsilon>0$ independent of $N$ such that for all $1\leq i \leq j\leq N$ and for all $2\leq k_{1},k_{2}\leq cM$, we have 
\begin{equation}\label{moment condition diluted wigner}
    N|\E (x^{N}_{i,j})^{k_{1}}(\overline{x^{N}_{i,j}})^{k_{2}}|\leq \frac{\log(N)^{h(k_{1}+k_{2})}}{N^{\varepsilon}},
\end{equation}
and $N\E|x^{N}_{i,j}|^{2}=1$. Note that $M$ depends on $N$ so that Equation \eqref{moment condition diluted wigner} must be satisfied for $k_{1},k_{2}$ possibly dependent on $N$.

These matrices are a special case of heavy Wigner matrices (see \cite{zakharevich_generalization_2006}) since at fixed $k$ the quantity in \eqref{moment condition diluted wigner} goes to 0 but we also impose a bound of order $N^{c}$ for the moments of order $\log N/\log \log N$. Actually, heavy Wigner matrices with moments conditions similar to \eqref{moment condition diluted wigner} also satisfy Assumption \ref{H} however, they do not satisfy in general Assumption \ref{assumption frobenius}.

Take $T_{N}$ as in Assumption \ref{H} of size at most $cM$. For $\bar e\in \bar E$, recall that $\bar e$ is the set of all edges that are between two vertices, say $v$ and $w$. Choose an arbitrary orientation, say from $v$ to $w$ and, set $m_{+}(\bar e)$, resp. $m_{-}(\bar e)$, the number of edges in the equivalence class $\bar e$ that have the same orientation, resp. the opposite one. With similar technique as in the previous case, we obtain
\begin{align*}
    \E \frac{1}{N^{K(N)}}\mathrm{Tr}^{0}(T_{N}(X_{N}))=N^{|V|-K(N)-|\bar E|}\prod_{\bar e \in \bar E}\E\left(NX_{N}(\phi_{N}(\bar e))^{m_{+}(\bar e)}\overline{X_{N}(\phi_{N}(\bar e))}^{m_{-}(\bar e)}\right).
\end{align*}

We denote by $\bar E_{2}\subset \bar E$ the set of edges such that $m_{+}(\bar e)=m_{-}(\bar e)=1$, i.e. edges that will contribute for one since $N\E|x|^{2}=1$. Each of the other edges $\bar e\in\bar E\setminus \bar E_{2}$ will have a contribution of $\log_{N}^{h(m_{+}(\bar e)+m_{-}(\bar e))}/N^{\varepsilon}$. Denoting now $E_{2}\subset E$ the set of edges such that $m_{+}(\bar e)=m_{-}(\bar e)=1$, we obtain
\begin{equation}
    \left|\E \frac{1}{N^{K(N)}}\mathrm{Tr}^{0}(T_{N}(X_{N}))\right|\leq N^{|V|-|\bar E|-K(N)}\frac{\log(N)^{h(|E|-2|\bar E_{2}|)}}{N^{\varepsilon(|\bar E|-|\bar E_{2}|)}}\leq N^{hc}.
\end{equation}

\paragraph{Entry wise product.} 
Let $X_{N}$ be a matrix that verifies Assumption \ref{H} at speed $M$ for some constant $h$ and let $\Gamma_{N}$ be a matrix that have all entries bounded by $C=C(N)$ such that there exists a constant $h'$ for which $M\log C\leq h'\log N\Leftrightarrow C^{M}\leq N^{h'}$ (note that this is always the case for $C=\log (N)^{\alpha}$ for any constant $\alpha$). Denote $W_{N}:=X_{N}\circ\Gamma_{N}$ the entry wise product of the two matrices. We easily have
\begin{equation}
    \left|\E\left[\frac{1}{N^{K(N)}}\mathrm{Tr}^{0}\left(T_{N}(X_{N}\circ\Gamma_{N})\right)\right]\right|\leq\left|\E\left[\frac{1}{N^{K(N)}}\mathrm{Tr}^{0}\left(T_{N}(X_{N})\right)\right]\right|\times C^{cM}=O\left( N^{c(h+h')}\right),
\end{equation}
showing that $W_{N}$ also verifies Assumption \ref{H}.

    \subsection{Assumption on the Frobenius norm}
    \begin{assumption}\label{assumption frobenius}
    Let $M=M(N)$ be a sequence that goes to infinity with $N$. The sequence of random matrices $X_{N}$ satisfies Assumption \ref{assumption frobenius} at speed $M$ if there exists a positive constant $C$ independent of $N$ such that for all $N$, 
    \begin{equation}\label{equation assumption frobenius}
        \E\frac{1}{N}||X^{M}||_{F}^{2}\leq C^{M}.
    \end{equation}
\end{assumption}

This is obviously true at any speed for sequences of matrices that are uniformly bounded in operator norm since $\frac{1}{N}||X_{N}^{M}||_{F}^{2}\leq ||X_{N}^{M}||^{2}_{op}\leq ||X_{N}||_{op}^{2M}\leq C^{2M}$ where $C$ is the bound of the operator norm of $X_{N}$.

However, this might also hold even though the operator norm of $X_{N}$ diverges. Take for instance the re-centered re-normalized version of the adjacency matrix of an Erdos-Renyi graph of parameter $p$ with $Np\rightarrow\infty$ and $Np/\log N\rightarrow 0$ as $N\rightarrow\infty$. It is shown in \cite{benaych-georges_largest_2019} that the largest eigenvalue of $X_{N}$ diverges. However, developing the trace in the Frobenius norm in term of injective trace and using calculations from the previous subsection we have the expression.
\begin{align*}
    \E\frac{1}{N}\mathrm{Tr}(X^{2M})= \E\frac{1}{N}\mathrm{Tr}(T_{2M}(X))=\sum_{\pi\in P(V_{2M})}N^{|V_{\pi}|-|\bar E_{\pi}|-1}(Np)^{|\bar E_{\pi}|-|E_{\pi} |/2}(1+o(1)),
\end{align*}
where the test-graph $T_{2M}=(V_{2M},E_{2M})$ is a cycle of length $2M$. When the partition $\pi$ makes $T^{\pi}$ into a double tree, then the contribution is 1. When the partition $\pi$ makes $T^{\pi}$ a fat tree (i.e. there is some edge with multiplicity at least 4) then the contribution is at most $(Np)^{-1}$ and otherwise the contribution is at most $1/N$. Bounding the number of partition by $2M^{2M}$ we obtain
\begin{align*}
    \E\frac{1}{N}\mathrm{Tr}(X^{2M})\leq C_{M}+M^{M}(Np)^{-1}+M^{M} N^{-1},
\end{align*}
where $C_{M}$ is the $M$-th Catalan number, i.e. the number of trees with $M$ edges (or double trees with $2M$ edges). The Catalan number can be bounded by some power of $M$ whereas, taking $M=\log(Np)/\log\log(Np)$, we obtain Equation \eqref{equation assumption frobenius}.

On the contrary, taking a sparse Erdos-Renyi graph, i.e. $Np\rightarrow c>0$, we would have had a contribution of order 1 for each fat tree and the number of fat trees of size $M$ cannot be bounded by some power of $M$.

Similarly for diluted Wigner matrices, distinguishing between double trees that have a contribution of 1, fat trees that have at least an edge of multiplicity strictly greater than 2 or other cases, recalling that $\bar E_{\pi,2}$ is the set of edges of $\bar E_{\pi}$ that have multiplicity exactly 2, we have
\begin{align*}
    \left|\E\frac{1}{N}\mathrm{Tr}(T_{2M}(X))\right|&\leq \sum_{\pi\in P(V_{2M})}N^{|V_{\pi}|-|\bar E_{\pi}|-1}\frac{\log(N)^{h(|E_{\pi}|-2|\bar E_{\pi,2}|)}}{N^{\varepsilon|(\bar E_{\pi}|-|\bar E_{\pi,2}|)}},\\
    &\leq  C_{M}+ M^{M}\log(N)^{2hM}N^{-\varepsilon}+M^{M}\log(N)^{2hM}N^{-1}.
\end{align*}
Take now $M=\beta\log N / \log\log N$, with $\beta< \frac{\varepsilon}{1+2h}$, only the contribution of $C_{M}$ remains and we obtain the wanted bound. Hence diluted Wigner matrices satisfy Assumption \ref{assumption frobenius} at speed $\beta \log N/\log\log N$ for $\beta$ small enough.

\section{Proof of the main theorems}
    \subsection{Preliminary results}

We provide some useful lemmas and notations prior to the proof of Theorems \ref{thm A} and \ref{thm B}. The first lemma ensures that when we glue two graphs together while not identifying vertices from a same graph, if we start with a graph that has a graph of colored components which is not a tree, the resulting graph will also have a graph of colored components which is not a tree. Recall from \eqref{Euler} that, for a simple connected graph $H=(V,E)$, $\eta(H):=|V|-|E|-1\leq 0$, and the equality holds if and only if $H$ is a tree.

\begin{lemma}\label{GCC pas arbre}
    Let $S_{1}=(V_{1},E_{1},\gamma_{1},\varepsilon_{1}),S_{2}=(V_{2},E_{2},\gamma_{2},\varepsilon_{2})$ be two test graphs labeled by $\{a,b\}$ and let $\sigma$ be an amalgamation between $V_{1}$ and $V_{2}$, i.e. a partition of $V_{1}\sqcup V_{2}$ where all the blocks are either of size one, either of size two with one vertex from $V_{1}$ and one vertex from $V_{2}$. Assume moreover that $\sigma$ has at least one block of size two (hence $(S_{1}\sqcup S_{2})^{\sigma}$ is connected). Then, we have
        \begin{equation}
            \eta(\GCC((S_{1}\sqcup S_{2})^{\sigma}))\leq \min\{\eta(\GCC(S_{1})),\eta(\GCC(S_{2}))\}.
        \end{equation}
\end{lemma}
\begin{proof}
        
    The graph $(S_{1}\sqcup S_{2})^{\sigma}$ can be obtained by the following process. Start with $G_{0}=S_{1}$. Since $\sigma$ has a block of size two, there is a couple $(v,w)\in V_{1}\times V_{2}$ such that $v\overset{\sigma}{\sim}w$. Choose an arbitrary edge adjacent to $w$ in $S_{2}$: say $e=(w,w')$. If $w'$ is in a block of size one of $\sigma$, add a new vertex $w'$ to $G_{0}$ and an edge between $v$ and $w'$ with the same label as $e$. Else, if there exists $v'\in V(G_{0})=V_{1}$ such that $v'\overset{\sigma}{\sim}w'$, then add an edge in $G_{0}$ between $v$ and $v'$ with the same label as $e$. The graph thus obtained is $G_{1}$. Repeat this process by choosing at step $l$, a new edge $e=(u,u')\in E_{2}$ that is adjacent to an edge chosen before (i.e. $u$ is a vertex of $G_{l}$) until there is no more edge in $E_{2}$. If $k$ denotes the number of edges of $S_{2}$ then $G_{k}=(S_{1}\sqcup S_{2})^{\sigma}$. We denote $\eta(l):=\eta(\GCC(G_{l}))$ and we show that this quantity does not increase. From Remark \ref{eta GCC}, recall that for any graph $G$ labeled by $A$ and $B$, $\eta(\GCC(G))=|\mathcal{V}_{cc}|-|\mathcal{V}_{co}|-1$. Hence, in order to keep track of $\eta(l)$ we only need to look at the different type of vertices added at each step.
    
    To go from $G_{l}$ to $G_{l+1}$ we distinguish two cases. Either we add an edge $e$ (labeled $a$ or $b$) (i)- between two already existing vertices $v_{1},v_{2}$ of $G_{l}$ or (ii)- between an already existing vertex $v_{1}$ of $G_{l}$ and a new vertex $v_{2}$. In other words, when an edge $(w,w')$ of $S_{2}$ has its two vertices in blocks of size two of $\sigma$, i.e. $\exists v,v'\in V_{1},\,v\overset{\sigma}{\sim}w,\,v'\overset{\sigma}{\sim}w'$, then we add an edge between $v$ and $v'$ in $G_{l}$ otherwise we add a new vertex and a new edge to it.

    We write 
    \begin{equation*}
        G_{l}=\bigsqcup_{i=1}^{I}T^{l,a}_{i}\cup\bigsqcup_{j=1}^{J}T_{j}^{l,b},
    \end{equation*}
    where the $T_{i}^{l,a}$ are the disjoint colored components labeled by $a$ of $G_{l}$ and the $T_{j}^{l,b}$ are the disjoint colored components labeled by $b$ of $G_{l}$. We write $V^{l,a}_{i}$ the set of vertices of $T_{i}^{l,a}$ (and similarly for $b$). For $i\in I$, we write $\overset{\circ}{V^{l,a}_{i}}\subset V_{i}^{l,a}$ the subset of vertices $v\in V_{i}^{l,a}$ that only have adjacent edges labeled $a$ and similarly for $b$. 

    Let $0\leq l\leq k-1$ us assume that $\eta(l)\leq\eta(0)=\eta(S_{1})$ and, without loss of generality, let us assume that the new edge $e=(v,w')$ we add is labeled $a$ (recall that $v$ is a vertex of $G_{l}$). If we are in the first scenario, let $v'\in V_{1}\subset V(G_{l})$ be such that  $v'\overset{\sigma}{\sim}w'$, we add an edge labeled $a$ between $v$ and $v'$.
    
    \begin{enumerate}
        \item If $v\in V_{i}^{l,a}$ and $v'\in V_{i}^{l,a}$, adding an edge labeled $a$ between $v$ and $v'$ does not change the graph of colored components since they are already in the same colored component (with labels $a$), hence $\eta(l+1)=\eta(l)$. Note that it does not matter whether $v$ or $v'$ are connectors or not.
        
        \item If $v\in V_{i_{1}}^{l,a}$ and $v'\in V_{i_{2}}^{l,a}$ with $i_{1}\neq i_{2}$. Adding an edge between $v$ and $v'$ will merge the colored components $T_{i_{1}}^{l,a}$ and $v'\in T_{i_{2}}^{l,a}$, hence the graph of colored components has one vertex less: $\eta(l+1)=\eta(l)-1$.
        
        \item If $v\in \overset{\circ}{V^{l,b}_{j_{1}}}$,
            \begin{enumerate}
                \item and if $v'\in\overset{\circ}{V^{l,b}_{j_{2}}}$ for some $j_{2}\in J$, in the graph of colored components, we add one vertex of type $cc$ and 2 vertices of type $co$ hence $\eta(l+1)=\eta(l)-1$;
                \item and if $v'\in V^{l,b}_{j_{2}}\cap V_{i}^{l,a}$ for some $j_{2}\in J$ and some $i\in I$, in the graph of colored components we add one vertex of type $co$, hence $\eta(l+1)=\eta(l)-1$;
                \item and if $v'\in\overset{\circ}{V^{l,a}_{i}}$ for some $i\in I$, in the graph of colored components we also add one vertex of type $co$, hence $\eta(l+1)=\eta(l)-1$.
            \end{enumerate}
    \end{enumerate}
    If we are in the second scenario: $w'$ is in a block of size one of $\sigma$. There are two cases:
    \begin{enumerate}
        \item either $v\in V_{i}^{l,a}$ for some $i\in I$, then the graph of colored components does not change, hence $\eta(l+1)=\eta(l)$; 
        
        \item or $v\in \overset{\circ}{V_{j}^{l,b}}$ for som $j\in J$, then we add one vertex of type $co$ and one vertex of type $cc$ in the graph of colored components, hence $\eta(l+1)=\eta(l)$.
    \end{enumerate}
        
    \begin{figure}[htbp]
    \centering
    % Première rangée
    \begin{subfigure}[b]{0.3\textwidth}
        \includegraphics[width=\textwidth]{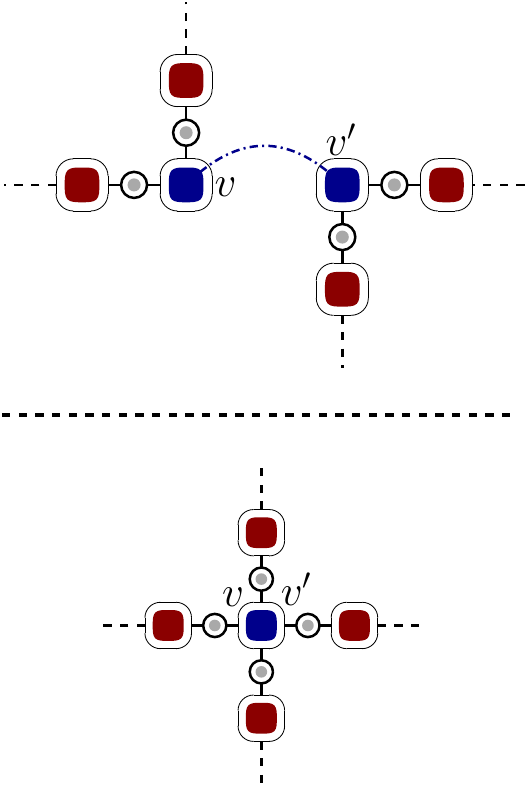}
        \caption{Case 2}
    \end{subfigure}
    \hfill
    \begin{subfigure}[b]{0.3\textwidth}
        \includegraphics[width=\textwidth]{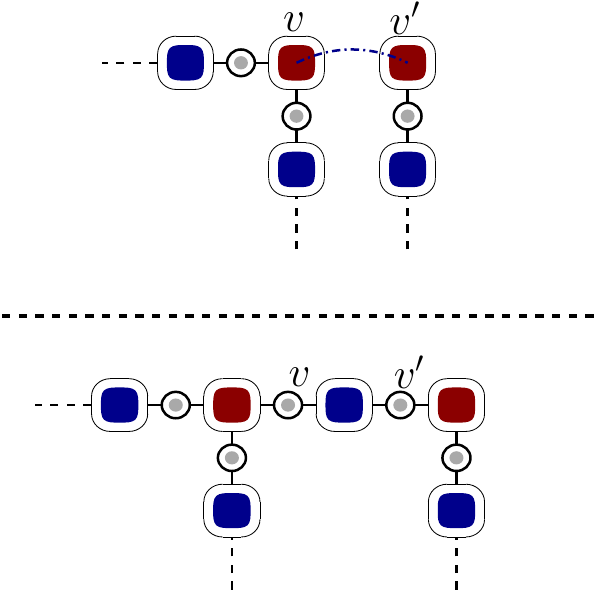}
        \caption{Case 3)a)}
    \end{subfigure}
    \hfill
    \begin{subfigure}[b]{0.3\textwidth}
        \includegraphics[width=\textwidth]{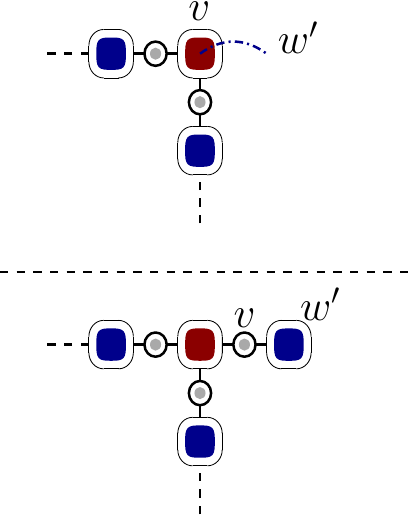}
        \caption{Scenario 2 Case 2}
    \end{subfigure}

    \caption{Examples of the modification of the $\GCC$ when adding an edge. The top picture is a detail of the $\GCC$ with a blue dashed line representing the edge we add at step $l$ in the graph. The bottom picture is the local change of the $\GCC$. Colored squares are vertex from colored components ($a$ is blue, $b$ is red) and gray circles are connectors.}
    \label{fig:six_images}
    \end{figure}
        
    Figure \ref{fig:six_images} provides several examples of different cases. After exhausting every case, we see that $(\eta(l))_{0\leq l \leq k}$ does not increase. Hence $\eta(k)\leq\eta(\GCC(S_{1}))$. This construction being symmetric, it concludes the proof.
    
\end{proof}

We now prove several results concerning the trace and injective trace of graph monomials in terms of $A+B$, $a+b$, $A$ and $B$.

For a graph monomial $h$ labeled by a single variable with $L$ connected components, we define $$t(h(X)):=\E\frac{1}{N^{L}}\sum_{\substack{\phi:V_{h}\rightarrow[N],\\\text{injective}}}\prod_{e\in E_{h}}X(\phi(e))$$ the quantity appearing in Assumption \ref{H}.

\begin{lemma}\label{lemma : injective trace A et B}
    Let $A$ and $B$ be random matrices that verify Assumption \ref{H} at speed $M$ with constants $h_{1}$ and $h_{2}$. Let $g$ be a graph monomial labeled by $a$ and $b$ with $m\leq cM$ vertices and edges and set $\eta:=\eta(\GCC(g))$. We have the following estimate :
    \begin{equation}
        \tau^{0}_{i}(g(A,B))=t(g_{a}(A))t(g_{b}(B))N^{\eta}+O(N^{\eta+c(h_{1}+h_{2})-1}m^{2}),
    \end{equation}
    where $g_{a}$ is the graph monomial obtained from $g$ after deleting all edges labeled $b$ and the isolated vertices and similarly for $g_{b}$.
\end{lemma}

The argument is made in \cite{male_traffic_2020} for test-graphs with bounded number of edges and vertices but it remains valid in our case. For completeness, we restate it here. 

\begin{proof}
Let $g=(V,E,\gamma,\varepsilon,v_{in}=v_{out})$ be a graph monomial labeled by two formal variables $a$ and $b$. Moreover since the matrices $A$ and $B$ are hermitian, we can omit $\varepsilon$ and we denote $X^{a}=A, X^{b}=B$. First note that since $A$ and $B$ are permutation invariant, $$\tau^{0}_{i}(g(A,B))= \frac{1}{N}\sum_{j=1}^{N}\tau^{0}_{j}(g(A,B)),$$ hence
\begin{align*}
    \tau_{i}^{0}(g(A,B))&=\frac{1}{N}\sum_{\substack{\phi:V\rightarrow[N],\\\phi\text{ injective}}}\E\left[\prod_{e=(v,w)\in E}X^{\gamma(e)}(\phi(w),\phi(v))\right]\\
    &=\frac{(N-1)!}{(N-|V|)!}\E_{\phi_{N}}\left[\prod_{e=(v,w)\in E}X^{\gamma(e)}(\phi_{N}(w),\phi_{N}(v))\right],
\end{align*}
where the expectation is also taken with respect to $\phi_{N}$, an injective map chosen uniformly among all injective maps from $V$ to $[N]$, independent of $A$ and $B$. Since the matrices $A$ and $B$ are invariant in law by permutation, we can add another 'layer' of random injective maps for one of these matrices, say $B$ as follows,
\begin{equation*}
    \tau_{i}^{0}(g(A,B))=\frac{(N-1)!}{(N-|V|)!}\E\left[\prod_{\substack{e\in E,\\\gamma(e)=a}}A(\phi_{N}(e))\prod_{\substack{e\in E,\\\gamma(e)=b}}B(\sigma\circ\phi_{N}(e))\right],
\end{equation*}
where $\sigma$ is a random uniformly chosen element of $\mathfrak{S}_{N}$ independent of $(\phi_{N},A,B)$. It is not hard to see that $(A,\phi_{N})$ is independent from $(B,\sigma\circ\phi_{N})$ and that $\sigma\circ\phi_{N}$ has same law as $\phi_{N}$. Hence we can separate the contribution from $A$ and from $B$. Let $g_{a}=(V_{a},E_{a})$, resp. $g_{b}=(V_{b},E_{b})$, be the graph monomial obtained from $g$ by keeping the edges labeled $a$, resp. $b$ and then deleting the isolated vertices. Denote $K_{a}$, resp. $K_{b}$ the number of connected components of $g_{a}$, resp. $g_{b}$. We now have
\begin{align*}
    \tau_{i}^{0}(g(A,B))&=\frac{(N-1)!}{(N-|V|)!}\E\left[\prod_{e\in E_{a}}A(\phi_{N}(e))\right]\E\left[\prod_{e\in E_{b}}B(\phi_{N}(e))\right]\\
    &=\frac{(N-|V_{a}|)!(N-|V_{b}|)!}{(N-|V|)!(N-1)!}N^{K_{a}-1}N^{K_{b}-1}t(g_{a}(A))t(g_{b}(B)).
\end{align*}

In order to use the second point of Lemma \ref{Stirling}, we only have to verify that $\eta:=K_{a}(N)+K_{b}(N)-1-|V_{a}|-|V_{b}|+|V|$ is non positive. Recall that the graph of colored components $\mathcal{GCC}(S):=(\mathcal{V},\mathcal{E})$ of $S$ has two type of vertices. The first type are the colored components: there are $K_{a}(N)+K_{b}(N)$ of them. The second type are the vertices from $S$ that belong to both a colored component in $a$ and a colored component in $b$, there are $|V_{a}\cap V_{b}|$ of them. Hence $|\mathcal{V}|=K_{a}(N)+K_{b}(N)+|V_{a}\cap V_{b}|$. Moreover, for each vertex of the second type, there are two corresponding edges in the graph of colored component, hence $|\mathcal{E}|=2|V_{a}\cap V_{b}|$. Finally, noting that $|V|-|V_{a}|-|V_{b}|=-|V_{a}\cap V_{b}|$, we see that $\eta=|\mathcal{V}|-|\mathcal{E}|-1$. Using Equation \eqref{Euler}, we obtain $\eta\leq 0$ with equality if and only if $\GCC(S)$ is a tree.

Finally, we get that 
\begin{equation*}
    \tau_{i}^{0}(g(A,B))=N^{\eta}t(g_{a}(A))t(g_{b}(B))\left(1+O\left(\frac{m^{2} }{N}\right)\right),
\end{equation*}
which concludes the proof using Assumption \ref{H}.
\end{proof}

We now compute the injective traffic distribution of the operators $a$ and $b$ in terms of injective distribution of the matrices $A$ and $B$. We denote $G_{N}=(V_{N},E_{N},\gamma_{N},\rho)$ the underlying $\rho$-rooted, labeled graph defined in Section 3, where $\rho\in[N]$ and where there is an edge between $v$ and $w$ if $\langle a\delta_{v},\delta_{w}\rangle\neq 0$ (then $\gamma_{N}((v,w))=a$) or $\langle b\delta_{v},\delta_{w}\rangle\neq0$ (then $\gamma_{N}((v,w))=b$). Let $g=(V,E,\gamma,\varepsilon,v_{in},v_{out})$ be a graph monomial labeled by two formal variables $a$ and $b$. Moreover since the matrices $A$ and $B$ are Hermitian, we can omit $\varepsilon$. It is easy to see that if $g$ has a $\GCC$ with respect to the families $a$ and $b$ that is not a tree, then there is no injective labeled graph morphism from $g$ to $G_{N}$, otherwise the graph of colored component of $G_{N}$ would not be a tree.

We denote 
    \begin{equation}
        \begin{split}
            \Phi:&V_{N}\longrightarrow[N]\\
            &\quad\rho\mapsto \rho,\\
            (j_{1},&\cdots,j_{n},\star)\mapsto j_{n},
        \end{split}
    \end{equation}
    where $\star$ stands for $a$ or $b$. The map $\Phi$ induces a labeled graph morphism from $G_{N}$ to the complete graph with $N$ vertices associated to the matrices $A$ and $B$ by keeping the labels as they are.

\begin{lemma}\label{injection petit graphe gros graphe}
    Let $g$ be a graph monomial with $v_{in}=v_{out}$ and such that its $\GCC$ is a tree. There is a bijection between the set of injective graph morphisms from $g$ to $G_{N}$ and maps from $V$ to $[N]$ that are injective on each colored components of $g$. This bijection is induced by the map $\Phi$.
\end{lemma}
\begin{proof}
    Let $\phi:V\rightarrow V_{N}$ be an injective graph morphism. The map $\Phi\circ\phi:V\rightarrow[N]$ is injective on each colored component of $g$. Indeed, let $S$ be a colored component, say of color $a$, of $g$. Since $\phi$ is an injective morphism, its image via $\phi$ is a subgraph of some $R_{a}(v)=\{(j_{1},\cdots,j_{n},b)\}\sqcup\{(j_{1},\cdots,j_{n+1},a),\,j_{n+1}\in[N]\setminus\{j_{n}\}\}$ (see Section 3). Since $\Phi$ is injective on $R_{a}(v)$, $\Phi\circ\phi$ is injective on $S$.

    Let now $\psi:V\rightarrow[N]$ be injective on each colored component of $g$ and denote $\CC\CC(g)$ the set of colored components of $g$. We define a function $h:\CC\CC(g)\rightarrow\N_{>0}$ which is a slight deformation of the graph-distance on $\GCC(g)$ as follows,
    \begin{equation}
        h(S)=\left\{
        \begin{split}
            \frac{1}{2}\left(d_{\GCC(g)}(S,v_{in})+1\right)\quad &\text{if }v_{in}\in \mathcal{V}_{co},\\
            \frac{1}{2}d_{\GCC(g)}(S_{in},S)+1,\quad&\text{if }v_{in}\in \overset{\circ}{S_{in}},
        \end{split}
        \right.
    \end{equation}
    where $\overset{\circ}{S_{in}}$ is the set of vertices of $S_{in}$ that are not connectors. Note that $h(S)$ is indeed a natural number and that if $v_{in}\in S$, $h(S)=1$. Moreover, $h$ induces a function, still denoted $h$ on $V$, by taking $h(v_{in})=0$ and $h(v)=\min\{h(S),S\ni v\}$ otherwise. Finally, for $v\neq v_{in}$ denote $c(v)\in \{a,b\}$, the color of the colored component that realizes the minimum, i.e. the color that gets $v$ closer to $v_{in}$. Since $\GCC(g)$ is a tree, for any vertex $v\neq v_{in}$, there exists a unique path $v_{in}=v_{0},v_{1},\cdots,v_{h(v)-1},v_{h(v)}=v$ such that for all $1\leq i \leq h(v)$, the vertices $v_{i-1}$ and $v_{i}$ belong to the same colored component, they are all connectors except possibly $v_{in}$ and $v$ and $h(v_{i})=i$. Denote $(j_{1},\cdots,j_{h(v)})=(\psi(v_{1}),\dots,\psi(v_{h(v)}))$ and define $\phi : V\rightarrow V_{N}$ by $\phi(v_{in})=\rho$ and $\phi(v)=(j_{1},\dots,j_{h(v)},c(v))$. This is well defined since $\psi$ is injective on each colored component and $v_{i-1},v_{i}$ belong to the same colored component and are two distinct vertices, we have $j_{i-1}\neq j_{i}$. Moreover, $\phi$ induces a graph morphism. Indeed, for any edge $e=(v,w)$, either $v=w$ and it is clear that $(\phi(v),\phi(w))$ is a n edge of $G_{N}$, either $v\neq w$. In that last case, since they are in the same colored component, $\phi(v)$ and $\phi(w)$ are both in the same colored component in $G_{N}$, hence $(\phi(v),\phi(w))$ is indeed an edge of $G_{N}$. Finally $\phi$ is injective : if $(\psi(v_{1}),\dots,\psi(v),c(v))=(\psi(w_{1}),\cdots,\psi(w),c(w))$, then $v$ and $w$ are in the same colored component and using the injectivity of $\psi$ on this colored component, we recursively show that $v=w$.
\end{proof}

Recall that a partition $\sigma$ of the vertices of a graph monomial $g$ is said to be an amalgamation between the colored components, written $\sigma\in P_{\#}(\CC\CC(g))$, if the blocs of $\sigma$ contain at most one vertex from each colored components. In other words, it may identify vertices from different colored components but never two vertices of the same one. We denote $1$ the special amalgamation consisting in only singletons. For any partition of the vertices of a graph monomial, recall that we denote $g^{\sigma}$ the graph monomial obtained after identifying vertices in a same block.

\begin{lemma}\label{lemma : trace injective a et b}
    Let $A$ and $B$ be random matrices and $a,b$ the associated operators constructed in Section 3. Let $g$ be a graph monomial labeled in two variables such that its $\GCC$ is a tree. We have
    \begin{equation}
        \tr^{0}_{\rho}(g(a,b))=\tr^{0}_{\rho}(g(A,B))+\sum_{\substack{\sigma\in P_{\#}(\CC\CC(g)),\\\sigma\neq 1}}\tr_{\rho}^{0}(g^{\sigma}(A,B)).
    \end{equation}
\end{lemma}
\begin{proof}
    With the definition of $\Phi$ and recalling that the operators $a$ and $b$ are defined through the last index of the vertex, we have
    \begin{equation*}
        \tr^{0}_{\rho}(g(a,b))=\sum_{\substack{\phi:V\rightarrow V_{N},\\\text{ injective graph morphism},\\\phi(v_{in})=\rho}}
        \prod_{e\in E}X^{\gamma(e)}(\Phi(\phi(e))),
    \end{equation*}
    where $X^{a}=A$ and $X^{b}=B$. Using Lemma \ref{injection petit graphe gros graphe} we have
    \begin{equation*}
        \tr^{0}_{\rho}(g(a,b))=\sum_{\phi}
        \prod_{e\in E}X^{\gamma(e)}(\phi(e)),
    \end{equation*}
    where the sum over $\phi$ now runs through all maps $\phi:V\rightarrow[N]$ that are injective on each colored component and such that $\phi(v_{in})=\rho$.
    
    A map $\phi:V\rightarrow[N]$ injective on each colored component can be seen as an injective map $\psi:V^{\sigma}\rightarrow[N]$ where $\sigma(i)\sim\sigma(j)\Leftrightarrow\phi(i)=\phi(j)$, is an amalgamation of the colored components of $g$. Hence
    \begin{align*}
        \tr^{0}_{\rho}(g(a,b))&=\sum_{\sigma\in P_{\#}(\CC\CC(g))}\sum_{\substack{\psi:V^{\sigma}\rightarrow[N],\\\text{injective},\\\psi(v_{in})=\rho}}
        \prod_{e\in E}X^{\gamma(e)}(\phi(e)),\\
        &=\sum_{\sigma\in P_{\#}(\CC\CC(g))}\tr_{\rho}^{0}(g^{\sigma}(A,B))
    \end{align*}
    
    The special case $\sigma = 1 = \{\{v_{1}\},\cdots,\{v_{m}\}\}$ does not change $g$ and we get the result.
   
\end{proof}

\begin{corollary}\label{cor : comparaison moment a+b et A+B}
    Let $m\leq cM$, let $A$ and $B$ be two matrices that satisfy Assumption \ref{H} at speed $M$, then
    \begin{equation}\label{eq : comparaison moment a+b et A+B}
        \langle(a+b)^{m}\delta_{\rho},\delta_{\rho}\rangle-(A+B)^{m}_{\rho,\rho}=O(2^{m}m^{2m}N^{c(h_{1}+h_{2})-1}).
    \end{equation}
\end{corollary}

\begin{proof}
    For $f:[m]\rightarrow\{a,b\}$, call $g_{f}$ the graph monomial which is a directed cycle whose edges are labeled by $f$ along the cycle. Expanding the moment by multi-linearity on the edges, we obtain
    \begin{align*}
        \langle(a+b)^{m}\delta_{\rho},\delta_{\rho}\rangle&=\sum_{f:[m]\rightarrow\{a,b\}}\sum_{\substack{\pi\in P(g_{f}),\\\GCC(g_{f}^{\pi})=\text{tree}}}\mathrm{Tr}_{\rho}^{0}(g_{f}^{\pi}(a,b)),\\
        &=\sum_{f:[m]\rightarrow\{a,b\}}\sum_{\substack{\pi\in P(g_{f}),\\\GCC(g_{f}^{\pi})=\text{tree}}}[\mathrm{Tr}_{\rho}^{0}(g_{f}^{\pi}(A,B)))+\sum_{\substack{\sigma\in P_{\#}(\CC\CC(g_{f}^{\pi})),\\\sigma\neq 1}}\tr_{\rho}^{0}((g_{f}^{\pi})^{\sigma}(A,B))].\\
    \end{align*}
    Hence, the left-hand side on Equation \eqref{eq : comparaison moment a+b et A+B} is equal to
    \begin{align*}
        \sum_{f:[m]\rightarrow\{a,b\}}\left(\sum_{\substack{\pi\in P(g_{f}),\\\GCC(g_{f}^{\pi})=\text{tree}}}\sum_{\substack{\sigma\in P_{\#}(\CC\CC(g_{f}^{\pi})),\\\sigma\neq 1}}\tr_{\rho}^{0}((g_{f}^{\pi})^{\sigma}(A,B))-\sum_{\substack{\pi\in P(g_{f}),\\\GCC(g_{f}^{\pi})\neq\text{tree}}}\tr_{\rho}^{0}(g_{f}^{\pi}(A,B))\right).
    \end{align*}
    Bound now each injective trace by $N^{c(h_{1}+h_{2})}N^{\eta}$ using Assumption \ref{H} and Lemma \ref{lemma : injective trace A et B}. Since we are only summing over partitions that have their graph of colored components that are not trees, we have $\eta\leq -1$ for each term. Bounding the number of partition we are summing over by $m^{m}$, we obtain Equation \eqref{eq : comparaison moment a+b et A+B}.
\end{proof}
    \subsection{Proof of Theorem \ref{thm A}}

Throughout the proof, we write $G_{A+B}(z):=G_{A+B}(zI)$ and similarly for $G_{a+b}$.
    Let $(\A,\D_{N}(\C),E)$, be a $C^{*}$-probability space, let $X\in\A$ be self-adjoint and let $z\in \C$ be a complex number with strictly positive imaginary part. We choose $M$ the minimum speed at which the matrices $A$ and $B$ satisfy Assumptions \ref{H} and \ref{assumption frobenius} and set $c>0$ to be chosen later. We have the identity
    \begin{equation}
        G_{X}(z)=E((z-X)^{-1})=\sum_{n=0}^{cM-1}\frac{1}{z^{n+1}}E(X^{n})+\frac{1}{z^{cM}}E\left(X^{cM}(z-X)^{-1}\right).
    \end{equation}
    We call $I_{X}^{M}(z)$ the first term and $R_{X}^{M}(z)$ the second on the right-hand side of the above equality. We apply this result to $X=A+B$ with $E=\Delta$ and to $X=a+b$ with $E=E_{\D_{N}(\C)}$ defined in Section 2. 
    We first expand the norm of $I^{M}(z):=I^{M}_{a+b}(z)-I^{M}_{A+B}(z)$ as follows
    \begin{equation}\label{eq : somme partielle etape 1}
        \begin{split}
            ||I^{M}(z)||_{F}^{2}=\sum_{\rho=1}^{N}\sum_{n,m=0}^{cM-1}&\frac{1}{z^{m+1}\bar z ^{n+1}}(\langle g_{m}(a+b)\delta_{\rho},\delta_{\rho}\rangle-g_{m}(A+B)_{\rho,\rho})\times\\
            &(\langle g_{n}(a+b)\delta_{\rho},\delta_{\rho}\rangle-g_{n}(A+B)_{\rho,\rho}),
        \end{split}
    \end{equation}
    where $g_{m}$ denotes an oriented cycle of length $m$ labeled by a single variable. Hence $\tr g_{m}(X)=\tr X^{m}$. We reproduce the proof of Corollary \ref{cor : comparaison moment a+b et A+B} to obtain
    \begin{equation}
        \begin{split}
            \langle g_{m}(a+b)\delta_{\rho},\delta_{\rho}\rangle-g_{m}(A+B)_{\rho,\rho}=&\sum_{f_{1}:[m]\rightarrow\{a,b\}}\sum_{\substack{\pi_{1}\in P(g_{f_{1}}),\\\GCC(g_{f_{1}}^{\pi_{1}})=\text{tree}}}\sum_{\substack{\sigma_1\in P_{\#}(\CC\CC(g_{f_1}^{\pi_1})),\\\sigma_1\neq 1}}\tr_{\rho}^{0}((g_{f_1}^{\pi_1})^{\sigma_1}(A,B))\\
            &-\sum_{\substack{\pi_1\in P(g_{f_1}),\\\GCC(g_{f_1}^{\pi_1})\neq\text{tree}}}\tr_{\rho}^{0}(g_{f_1}^{\pi_1}(A,B)),
        \end{split}  
    \end{equation}
    and similarly for the $n$-th moment replacing all 1 subscripts by 2's.
    Note that for any $\pi_{1},\sigma_{1}$ in the last sums, the graph of colored components of $(g_{f_{1}}^{\pi_{1}})^{\sigma_{1}}$ is not a tree, hence $\eta(\GCC((g_{f_{1}}^{\pi_{1}})^{\sigma_{1}}))\leq -1$. Hence when expanding the product in Equation \eqref{eq : somme partielle etape 1}, we will have a sum of expressions of the form
    \begin{align*}
        \tr_{\rho}^{0}(h(A,B))\tr_{\rho}^{0}(h'(A,B)),
    \end{align*}
    for some graph monomials $h,h'$ that have their graph of colored components which is not a tree and have the same input and output vertex. We express this product as a sum of injective trace of some modified graph monomial. Set $h\cdot h'$ the graph monomial obtained after identifying the input vertex of $h$ with that of $h'$, so that we have
    \begin{align*}
        \tr_{\rho}^{0}(h(A,B))\tr_{\rho}^{0}(h'(A,B))=\sum_{\sigma\in P_{\#}(h,h')}\tr^{0}_{\rho}((h\cdot h')^{\sigma}(A,B)).
    \end{align*}
    Using Lemma \ref{GCC pas arbre}, we know that $\eta(\GCC((h\cdot h')^{\sigma}))\leq -1$. We use Lemma \ref{lemma : injective trace A et B} to bound the expectation of each injective trace that appear in the sums by $N^{c(h_{1}+h_{2})-1}$. Finally, we bound the number of partitions of a set with $k$ elements by $k^{k}$ and we obtain
    \begin{align*}
        \E||I^{M}(z)||_{F}^{2}&\leq\sum_{\rho=1}^{N}\sum_{n,m=0}^{cM-1}\sum_{\substack{f_{1}:[m]\rightarrow\{a,b\},\\f_{2}[n]\rightarrow\{a,b\}}}\frac{1}{|z|^{n+m+2}}n^{2n}m^{2m}(n+m)^{n+m}N^{c(h_{1}+h_{2})-1},\\
        &\leq \frac{N}{|z|^{2}}\left(\sum_{m=0}^{cM-1}\left(\frac{2(cM)^{3}}{|z|}\right)^{m}\right)^{2}N^{c(h_{1}+h_{2})-1},\\
        &\leq \frac{4N}{|z|^{2}}\left(\frac{2c^{3}}{|z|}\right)^{2cM}M^{6cM}N^{c(h_{1}+h_{2})-1},
    \end{align*}
    where for the last line, we bounded each term in the sum by $\left(\frac{2(cM)^{3}}{|z|}\right)^{cM-1}$ which is valid for $M$ large enough, hence $N$ large enough. Using $M^{6cM}\leq N^{6c}$, we have for $|z|\geq 2c^{3}$,
    \begin{equation}\label{eq : borne somme partielle}
        \E\frac{1}{N}||I^{M}(z)||_{F}^{2}\leq\frac{1}{|z|^{2}} N^{c(h_{1}+h_{2}+6)-1}e^{2cM\log(2c^{3}/|z|)}.
    \end{equation}

We now bound the norm of $R^{M}_{A+B}(z)$. Recall that the Fröbenius norm is the Schatten-2 norm and that the operator norm is the Schatten-$\infty$ norm. Hölder's inequality on Schatten norms yields
    \begin{align*}
        ||R_{A+B}^{M}(z)||_{F}\leq\frac{1}{|z|^{cM}}||(A+B)^{cM}||_{F}||(z-A-B)^{-1}||_{op}.
    \end{align*}
Since $A$ and $B$ are Hermitian, denoting $||X||_{p}:=\mathrm{Tr}((XX^{*})^{p})^{1/2p}$ the Schatten $p$-norm of a matrix $X$ and using Minkowski's inequality, then Hölder's inequality, we have
\begin{align*}
    ||(A+B)^{cM}||^{2}_{F}=||A+B||_{cM}^{2cM}&\leq(||A||_{cM}+||B||_{cM})^{2cM},\\
    &\leq\left(||A^{cM}||_{F}^{1/cM}+||B^{cM}||_{F}^{1/cM}\right)^{2cM},\\
    &\leq 2^{2cM-1}(||A^{cM}||^{2}_{F}+||B^{cM}||^{2}_{F}).
\end{align*}
Since $A$ and $B$ satisfy Assumption \ref{assumption frobenius} at speed $cM$ with constant $C$, we finally obtain
\begin{equation}\label{eq : borne reste A+B}
    \E \frac{1}{N}||R_{A+B}^{M}(z)||_{F}^{2}\leq\left(\frac{4C}{|z|^{2}}\right)^{cM}\frac{1}{|\Im(z)|^{2}}.
\end{equation}

To estimate $R_{a+b}^{M}(z)$, recall from the end of Section 3 that $\varphi:\A\rightarrow\C,\,T\mapsto \frac{1}{N}\sum_{v\in R_{a}(x)}\langle T\delta_{v},\delta_{v}\rangle$, is positive and makes $(\A,\varphi)$ into a $C^{*}$-probability space. Moreover, setting $x:=(z-a-b)^{-1}(a+b)^{cM}$ we can write
\begin{align*}
    \frac{1}{N}||R_{a+b}^{M}(z)||_{F}^{2}=\frac{1}{|z|^{2cM}}\varphi(E_{\D_{N}(\C)}(x)E_{\D_{N}(\C)}(x)^{*})\leq\frac{1}{|z|^{2cM}}\varphi(xx^{*}),
\end{align*}
where we used that $E_{\D_{N}(\C)}$ is a $*$-homomorphism and is also a projection: one may compute explicitly $\langle x\delta_{v},x\delta_{v}\rangle = \langle (x-p(x)+p(x))\delta_{v},(x-p(x)+p(x))\delta_{v}\rangle $ for any $*$-homomorphism $p$ which is a also a projection to obtain the inequality above. Now, since $a+b$ is self-adjoint, there exists a compactly supported measure $\mu$ such that 
\begin{align*}
    \varphi(xx^{*})=\int(z-t)^{-1}t^{2cM}(\bar z -t)^{-1}\dd\mu (t)\leq \frac{1}{|\Im(z)|^{2}}\frac{1}{N}\sum_{\rho=1}^{N}\langle(a+b)^{2cM}\delta_{\rho},\delta_{\rho}\rangle.
\end{align*}
Using Corollary \ref{cor : comparaison moment a+b et A+B}, we have,
\begin{equation}\label{eq : bound reste a+b}
    \begin{split}
        \E\frac{1}{N}||R_{a+b}^{M}(z)||_{F}^{2}&\leq \frac{1}{|z|^{2cM}|\Im(z)|^{2}}\left(\frac{1}{N}||(A+B)^{cM}||_{F}^{2}+O(2^{cM}(cM)^{2cM}N^{c(h_{1}+h_{2})-1})\right),\\
        &\leq \left(\frac{1}{|z|^{2}}\right)^{cM}\left(\frac{(4C)^{cM}}{|\Im(z)|^{2}}+(2c^{2})^{cM}N^{c(h_{1}+h_{2}+2)-1}\right),
    \end{split}
\end{equation}
where we used for the last line the fact that $A$ and $B$ satisfy Assumption \ref{assumption frobenius} at speed $M$ with constant $C$ and that $M^{2M}\leq N^{2c}$.

Putting everything together, for $0< c < \frac{1}{h_{1}+h_{2}+6}$ - this ensures that the terms with $N^{c(h_{1}+h_{2}+p)-1}$ for $p=2$ (Equation \eqref{eq : bound reste a+b}) and $p=6$ (Equation \eqref{eq : borne somme partielle}) go to zero faster than the terms in $\alpha^{M}$ for some constant $0<\alpha<1$ (Equations \eqref{eq : bound reste a+b},\eqref{eq : borne reste A+B})-  and $|z|^{2}\geq \max\{4C, 2c^{2},2c^{3}\}=\max\{4C, 2c^{2}\}=:r_{0}^{2}$, we obtain
\begin{equation}
    \E\frac{1}{N}||G_{a+b}(z)-G_{A+B}(z)||_{F}^{2}\leq \frac{4}{|z|^{2}}\left(\frac{r_{0}}{|z|}\right)^{2cM}
\end{equation}
    
    \subsection{Proof of Theorem \ref{thm B}}
    Note that Theorem \ref{thm A} already includes sparse bounded matrices. We improve the bound with a slightly different proof. The proof still relies on a moment method but we now show that the moments of $A+B$ and the moments of $a+b$ coincide up to order $\log N$ with high probability.

\paragraph{Graph-related definitions}
Given a graph $G=(V,E)$, we define, for $x,y\in V$, $d_{G}(x,y)$ to be the length of the shortest path from $x$ to $y$, where the length of a path is its number of edges. In the following, if the graph $G$ is $A+B$ we will simply denote its distance by $d$. Let $x\in V$, $k\in\N$ and denote
\begin{align*}
    B_{G}(x,k)&:=\{y\in V, d_{G}(x,y)\leq k\},\quad\quad B_{G}(k):=\max_{x\in V}\{|B_{G}(x,k)|\}\\
    C_{G}(x,k)&:=\{y\in V, d_{G}(x,y)=k\},\quad\quad C_{G}(k):=\max_{x\in V}\{|C_{G}(x,k)|\}.
\end{align*}

Note that for $C$-sparse matrices, we always have 
\begin{equation}\label{bound on the size of the ball}
    |B_{A}(x,n)|\leq C^{n}.
\end{equation}

For $G=A+B$, for all $x\in V=[N]$, we denote
\begin{equation*}
    R(x):=\max\{k\in\N,\, \mathcal{GCC}(B(x,k))\text{ is a tree}\}.
\end{equation*}
Finally, we set 
\begin{equation*}
    F_{N}(n):=\{x\in[N], R(x)\geq n \}.
\end{equation*}
The quantity $R(x)$ has to be thought of as the girth of $x$ with respect to the colored components: it is the maximal radius such that, locally, the ball around $x$ satisfies the condition that its graph of colored components is a tree. For the graph associated to $a+b$ constructed in Section 3, the girth is infinite for all vertices since the graph of colored component is a tree.

\begin{lemma}\label{Markov bound}
    Let $A_{N}$ and $B_{N}$ be $C$-sparse matrices, we have 
    \begin{equation}
        \E\left(|F_{N}(\kappa\log N)|\right)\leq(eC^{2})^{n}.
    \end{equation}
\end{lemma}

\begin{proof}
    Throughout the proof, we let $n$ be an integer smaller than $\kappa \log N$ for some constant $\kappa$ to be chosen later. Notice that
    \begin{equation*}
        \E\left(\left|F_{N}(n)^{c}\right|\right)=\sum_{x=1}^{N}\P\left(\mathcal{GCC}(B(x,n))\text{ is not a tree}\right).
    \end{equation*}

    Let $x_{1},x_{2}\in[N], $ and $2\leq k,l \leq n$ and set 
    \begin{equation*}
        f_{2l}^{a\rightarrow b}(x_{1},x_{2},k):=\P\left(\exists\gamma_{1},\cdots,\gamma_{2l}, \sum_{i=1}^{2l}\ell(\gamma_{i})\leq k,\,x_{1}\underset{\gamma_{1}}{\overset{a}{\sim}}y_{1}\underset{\gamma_{2}}{\overset{b}{\sim}}\cdots \underset{\gamma_{2l-1}}{\overset{a}{\sim}}y_{2l-1}\underset{\gamma_{2l}}{\overset{b}{\sim}}x_{2}\right),
    \end{equation*}
    where $\ell(\gamma)$ is the length (i.e. the number of edges) of the path $\gamma$ in $A+B$ and where the notation $x_{1}\underset{\gamma_{1}}{\overset{a}{\sim}}y_{1}\underset{\gamma_{2}}{\overset{b}{\sim}}\cdots \underset{\gamma_{2l-1}}{\overset{a}{\sim}}y_{2l-1}\underset{\gamma_{2l}}{\overset{b}{\sim}}x_{2}$ stands for the condition that the path $\gamma_{2i+1}$ only uses edges from $A\cap B(x,n) $ and links $y_{2i}$ to $y_{2i+1}$, the path $\gamma_{2i}$ only uses edges from $B\cap B(x,n)$ and links $y_{2i-1}$ to $y_{2i}$, such that $\gamma:=\coprod\gamma_{i}$ does not use twice the same edge, with the convention that $y_{0}=x_{1}$ and $y_{2l}=x_{2}$.

    Hence, we obtain
    \begin{equation}\label{first bound}
        \E\left(\left|F_{N}(n)^{c}\right|\right)\leq\sum_{x=1}^{N}\sum_{y\in B(x,n)}\sum_{l=1}^{\lfloor n/2\rfloor}f^{a\rightarrow b}_{2l}(y,y,n).
    \end{equation}
    Note that we only took into account paths that go through an even number of colored components since a path $\gamma$ of the form $x_{1}\underset{\gamma_{1}}{\overset{a}{\sim}}y_{1}\underset{\gamma_{2}}{\overset{b}{\sim}}\cdots \underset{\gamma_{2l-2}}{\overset{b}{\sim}}y_{2l-2}\underset{\gamma_{2l-1}}{\overset{a}{\sim}}x_{2}$ can be understood as a path going through $2l$ colored components where the last $\gamma_{2l}$ is empty.
    
    Let us now bound $f^{a\rightarrow b}_{2l}(y,y,n)$ recursively in $l$.

    For $l=1$, we bound more generally the following quantity:
    \begin{align*}
        f^{a\rightarrow b}_{2}(x_{1},x_{2},n)&\leq\sum_{k=1}^{n-1}\P\left[\left(B_{A}(x_{1},k)\setminus\{x_{1}\}\right)\cap\left(B_{B}(x_{2},n-k)\setminus\{x_{2}\}\right)\neq\varnothing\right]\\
        &\leq\sum_{k=1}^{n-1}1-\frac{\binom{N-|B_{A}(x_{1},k)|+1}{|B_{B}(x_{2},n-k)|-1}}{\binom{N}{|B_{B}(x_{2},n-k)|-1}},
    \end{align*}
    where we used the fact that the matrices $A$ and $B$ being permutation invariant, we can conjugate $A$ and $B$ by two independent uniform permutation matrices and obtain the second line.

    Since $A$ and $B$ are sparse, the quantity $\frac{|B_{A}(n)||B_{B}(n)|}{N}$ goes to 0 as $N$ goes to infinity for $\kappa<(2\log C)^{-1}s$, see Equation \eqref{bound on the size of the ball}. Hence, using Stirling's formula, we get

    \begin{align*}
        f^{a\rightarrow b}_{2}(x_{1},x_{2},n)&\leq\sum_{j=1}^{n-1}\frac{|B_{A}(x_{1},j)||B_{B}(x_{2},n-j)|}{N}+O\left(n\left(\frac{|B_{A}(n)||B_{B}(n)|}{N}\right)^{2}\right)\\
        &\leq \frac{nC^{n}}{N}+O\left(\frac{nC^{2n}}{N^{2}}\right).
    \end{align*}

    We show by induction that 
    \begin{equation*}
        f_{2l}^{a\rightarrow b}(x_{1},x_{2},n)\leq \frac{n^{2l-1}C^{n}}{(2l-1)! N}(1+o(1)).
    \end{equation*}

    Notice now that for $2l+2\leq n$, 
    \begin{equation*}
        f_{2l+2}^{a\rightarrow b}(x_{1},x_{2},n)\leq\sum_{j_{1}=1}^{n}\sum_{y_{1}\in C_{A}(x_{1},j_{1})}\sum_{j_{2}=1}^{n-j_{1}}\sum_{y_{2}\in C_{B}(x_{2},j_{2})}f^{a\rightarrow b}_{2l}(y_{2},y_{1},n-j_{1}-j_{2}).
    \end{equation*}

    Hence, using the induction hypothesis, we get 

    \begin{align*}
        f_{2l+2}^{a\rightarrow b}(x_{1},x_{2},n)&\leq\sum_{j_{1}+j_{2}\leq n}C^{j_{1}}C^{j_{2}}\frac{(n-j_{1}-j_{2})^{2l-1}C^{n-j_{1}-j_{2}}}{(2l-1)! N}(1+o(1)),\\
        &\leq\frac{n^{2l+1}C^{n}}{(2l+1)! N}(1+o(1)),
    \end{align*}
    where we used the fact that
    \begin{align*}
        \sum_{i+j\leq n}(i+j)^{2l-1}&=n^{2l+1}\frac{1}{n^{2}}\sum_{(i+j)/n\leq 1}\left(\frac{i+j}{n}\right)^{2l-1}\\
        &=n^{2l+1}\int_{0\leq x+y\leq 1}(1-(x+y))^{2l-1}\dd x \dd y (1+o(1)),\\
        &=\frac{n^{2l+1}}{2l(2l+1)}(1+o(1)).
    \end{align*}
    
    Going back to Equation \ref{first bound}, we have
    \begin{equation*}
        \E\left(\left|F_{N}(n)^{c}\right|\right)\leq\sum_{x=1}^{N}\sum_{y\in B(x,n)}\sum_{l=1}^{\lfloor n/2\rfloor}\frac{n^{2l-1}C^{n}}{(2l-1)! N}\leq \sum_{x=1}^{N}\sum_{y\in B(x,n)}\frac{e^{n}C^{n}}{N}\leq(eC^{2})^{n}.
    \end{equation*}
\end{proof}

\begin{proof}[Proof of Theorem \ref{thm B}]
    Let $0<\kappa<1/2$ to be chosen later, let $\eta_{0}$ and $D$ as in Theorem \ref{thm B}. Writing $\Tilde{A}=D^{-1/2}AD^{-1/2}$ and $\Tilde{B}=D^{-1/2}BD^{-1/2}$, we obtain
    \begin{equation*}
        G_{A+B}(D)_{i,i}=D_{i,i}^{-1}\sum_{n=0}^{\infty}\left(\Tilde{A}+\Tilde{B}\right)_{i,i}^{n}=D_{i,i}^{-1}\sum_{n=0}^{\kappa\log N}\left(\Tilde{A}+\Tilde{B}\right)_{i,i}^{n}+D_{i,i}^{-1}\sum_{n=\kappa\log N}^{\infty}\left(\Tilde{A}+\Tilde{B}\right)_{i,i}^{n}.
    \end{equation*}

    The last term on the right-hand side is bounded as follows
    \begin{align*}
         \left|D_{i,i}^{-1}\sum_{n=\kappa\log N}^{\infty}\left(\Tilde{A}+\Tilde{B}\right)_{i,i}^{n}\right|&\leq \frac{\left(\frac{||A||_{op}+||B||_{op}}{d}\right)^{\kappa\log N}}{d-||A||_{op}-||B||_{op}},\\
        &\leq\frac{1}{d}N^{\kappa\log\left(\frac{2C}{d}\right)},
    \end{align*}
    where we recall that $d:=\inf\{|\Im D_{i,i}|,\,1\leq i \leq N\}$. The same goes for $G_{a+b}(D)$ since the operator norms of $a$ and $b$ are the same as those of $A$ and $B$ respectively

    Recall from Corollary \ref{free over diagonal matrices}, that the operator-valued measure $\mu_{a+b}:\D_{N}(\C)\langle\X\rangle\rightarrow\D_{N}(\C)$ describes the law of the free sum over the diagonal of $A$ and $B$. We denote $\overline{a+b}:=D^{-1/2}(a+b) D^{-1/2}$ where $D^{-1/2}$ is to be seen as a diagonal operator through the embedding of diagonal matrices into diagonal operators described in Section 3, and let $\mu_{\overline{a+b}}$ denote the associated operator-valued measure.

    Expanding both operator-valued Cauchy transforms, we obtain
    \begin{align*}
        \E\frac{1}{N}||G_{A+B}(D)-G_{a+b}(D)||_{F}^{2}&\leq \frac{2}{d^{2}}\E\frac{1}{N}\sum_{i=1}^{N}\left|\sum_{n=0}^{\kappa\log N}\left(\Tilde{A}+\Tilde{B}\right)_{i,i}^{n}-\langle(\overline{a+b})^{n}\delta_{i},\delta_{i}\rangle\right|^{2}\\
        &+4\frac{1}{d^{2}}N^{2\kappa\log\left(\frac{
        2C}{d}\right)}
    \end{align*}

    Note now that, looking at the proof of Corollary \ref{cor : comparaison moment a+b et A+B}, $i\in F_{N}(\kappa\log N)\implies \left(\Tilde{A}+\Tilde{B}\right)_{i,i}^{n}=\langle(\overline{a+b})^{n}\delta_{i},\delta_{i}\rangle,\,\forall 1\leq n \leq \kappa\log N$. Indeed, Corollary \ref{cor : comparaison moment a+b et A+B} is a statement on each realization of $a+b$ and $A+B$ and all the remaining terms are equal to 0 because they are injective trace of graph monomials that have their $\GCC$ which is not a tree. Hence we have,
    \begin{align*}
        \E\frac{1}{N}||G_{A+B}(D)-G_{a+b}(D)||_{F}^{2}\leq&\frac{2}{d^{2}}\E\left[\frac{|F_{N}(\kappa\log N)|^{c}}{N}\left(\frac{||A||_{op}+||B||_{op}}{d}\right)^{2\kappa\log N}\right]\\
        &+4\frac{1}{d^{2}}N^{2\kappa\log\left(\frac{2C}{d}\right)},\\
        \leq& \frac{2}{d^{2}}N^{\kappa\log(eC^{2})-1+2\kappa\log\left(\frac{2C}{d}\right)}+4\frac{1}{d^{2}}N^{2\kappa\log\left(\frac{2C}{d}\right)},
    \end{align*}
    where the last inequality comes from Lemma \ref{Markov bound}.
    Taking $\kappa= \frac{1}{1+2\log C}<1/2$, we obtain the wanted bound.
\end{proof}

\appendix

\section{An asymptotic development}
    \begin{lemma}[Two asymptotic developments]\label{Stirling}
\begin{enumerate}
    \item Let $V$ be a set of size dependent on $N$ such that $|V|^{2}/N\rightarrow0$ as $N\rightarrow\infty$.
    \begin{equation*}
        \frac{N!}{(N-|V|)!}=N^{|V|}\left(1+O\left(\frac{|V|^{2}}{N}\right)\right).
    \end{equation*}
    \item Let $V,V_{1},V_{2}$ be sets of size dependent on $N$ and $K_{1},K_{2}$ be numbers also dependent on $N$ with the following conditions:
    \begin{itemize}
        \item We have the inclusions $V_{1},V_{2}\subset V$ and we denote $\delta V:=V_{1}\cap V_{2}$.
        \item The ratio $|V|^{2}/N$ goes to $0$ as $N$ goes to infinity (hence it is also true for $V_{1}$ and $V_{2}$ instead of $V$).
        \item The quantity $\eta:=K_{1}+K_{2}-1-|V_{1}|-|V_{2}|+|V|$ is non positive.
    \end{itemize}
    Then, we have
    \begin{equation*}
        \Gamma_{N}:=N\frac{(N-|V_{1}|)!(N-|V_{2}|)!}{(N-|V|)!N!}N^{K_{1}-1}N^{K_{2}-1}=N^{\eta}\left(1+O\left(\frac{|V|^{2}}{N}\right)\right)
    \end{equation*}
\end{enumerate}
    
\end{lemma}
\begin{proof}
    For the first part, we use Stirling's equivalent formula:
    \begin{align*}
        \frac{N!}{(N-|V|)!}&=\frac{N^{N}e^{N-|V|}}{(N-|V|)^{N-|V|}e^{N}}\left(1-\frac{|V|}{N}\right)^{-1/2}\left(1+O\left(\frac{1}{N}\right)\right),\\
        &=N^{|V|}\exp\left(-\left(N-|V|+\frac{1}{2}\right)\log\left(1-\frac{|V|}{N}\right)-|V|\right)\left(1+O\left(\frac{1}{N}\right)\right),\\
        &=N^{|V|}\exp\left(-\frac{|V|^{2}}{N}+\frac{|V|}{2N}+\frac{|V|^{2}}{2N}+O\left(\frac{|V|^{3}}{N^{2}}\right)\right)\left(1+O\left(\frac{1}{N}\right)\right),\\
        &=N^{|V|}\left(1+O\left(\frac{|V|^{2}}{N}\right)\right),
    \end{align*}
    where the asymptotic development in the exponential is valid since all the terms are of the form $\frac{|V|^{k+1}}{N^{k}}$ which goes to 0 as $N$ goes to infinity by assumption.
    For the second point, we note that we already proved that 
    \begin{equation*}
        f_{\pm}(V):=\exp\left(\pm\left[\left(N-|V|+\frac{1}{2}\right)\log\left(1-\frac{|V|}{N}\right)+|V|\right]\right)=1+O\left(\frac{|V|^{2}}{N}\right),
    \end{equation*}
    and similarly for $V_{1},V_{2}$ since they all verify the assumptions of the first point. Again, using the asymptotic development of the factorial, we get
    \begin{align*}
        \Gamma_{N}&=N^{\eta}\frac{\left(1-\frac{|V_{1}|}{N}\right)^{N-|V_{1}|+1/2}\left(1-\frac{|V_{2}|}{N}\right)^{N-|V_{2}|+1/2}}{\left(1-\frac{|V|}{N}\right)^{N-|V|+1/2}}e^{|V_{1}|+|V_{2}|-|V|}\left(1+O\left(\frac{1}{N}\right)\right),\\
        &=N^{\eta}f_{+}(V_{1})f_{+}(V_{2})f_{-}(V)\left(1+O\left(\frac{1}{N}\right)\right)=N^{\eta}\left(1+O\left(\frac{|V|^{2}}{N}\right)\right).
    \end{align*}
\end{proof}

\bibliography{biblio.bib}{}
\bibliographystyle{plain}

\end{document}